\documentclass{amsart}

\usepackage{graphicx,enumerate}
\usepackage{amscd}
\usepackage{amsmath, amssymb, amsfonts}
\usepackage{tikz}

\newtheorem{theorem}{Theorem}[section]
\newtheorem{lemma}[theorem]{Lemma}

\newtheorem{proposition}[theorem]{Proposition}

\theoremstyle{definition}
\newtheorem{definition}[theorem]{Definition}

\theoremstyle{remark}
\newtheorem{remark}[theorem]{Remark}

\numberwithin{equation}{section}

\begin{document}

\title[A Kuratowski-Type Classification of Critical Complexes for the 3-Sphere]{A Kuratowski-Type Classification of Critical Complexes for the 3-Sphere}

\author{Mario Eudave-Mu\~{n}oz}
\address{Instituto de Matem\'aticas, Universidad Nacional Aut\'onoma de Mexico, Cuernavaca, Mor., MEXICO}
\email{mario@matem.unam.mx}
\thanks{The first author is partially supported by grant PAPIIT-UNAM IN117423}

\author{Makoto Ozawa}
\address{Department of Natural Sciences, Faculty of Arts and Sciences, Komazawa University, 1-23-1 Komazawa, Setagaya-ku, Tokyo, 154-8525, Japan}
\email{w3c@komazawa-u.ac.jp}
\thanks{Corresponding author: Makoto Ozawa.}
\thanks{The second author is partially supported by Grant-in-Aid for Scientific Research (C) (No. 17K05262), The Ministry of Education, Culture, Sports, Science and Technology, Japan}





\begin{abstract}
We give a Kuratowski-type classification of a graph-defined class of minimal
piecewise-linear obstructions to embeddability in the 3-sphere.  A finite
simplicial complex \(X\) is called critical for \(S^3\) if \(|X|\) does not
embed in \(S^3\), whereas deleting the open star of any simplex in the second
barycentric subdivision of \(X\) yields a polyhedron embeddable in \(S^3\).

The main theorem completely classifies critical complexes of the form
\((G\times S^1)\cup H\), where \(G\) and \(H\) are graphs and \(H\) is
attached along vertices of the branch set of \(G\times S^1\).  We prove that
there are exactly seven such complexes up to homeomorphism: two \(K_4\)-type
complexes, four \(\Theta_4\)-type complexes, and one \(K_{2,3}\)-type complex.

The proof is combinatorial in nature.  By collapsing the \(S^1\)-factor of
\(G\times S^1\), we associate to \(X\) a reduction graph
\(\widehat X=G\cup H\).  Criticality implies that \(H\) is a forest, \(G\) is
planar, and \(\widehat X\) is inclusion-minimal non-planar.  Kuratowski's
theorem therefore reduces the classification to the cases \(K_5\) and
\(K_{3,3}\).  A finite analysis of forest attachments, together with a
face-incidence criterion for embeddability, leaves precisely the seven models
listed above.

We also prove that every non-embeddable regular multibranched surface in
\(S^3\) contains a critical subcomplex of the form \(M\cup H\), where \(M\)
is a regular multibranched surface and \(H\) is a graph.
\end{abstract}

\maketitle

\section{Introduction}

Kuratowski's theorem characterizes non-planar graphs by two minimal
obstructions: a graph is non-planar if and only if it contains a subdivision
of \(K_5\) or \(K_{3,3}\) \cite{K}.  Equivalently, \(K_5\) and \(K_{3,3}\)
are the basic minimal obstructions to embeddability in the \(2\)-sphere.
This paper studies an analogous problem for finite polyhedra in the
\(3\)-sphere.

We work in the piecewise-linear category.  Throughout, \(X\) denotes a finite
simplicial complex and \(|X|\) its underlying polyhedron.  We call \(X\)
\emph{critical} for a simplicial complex \(Y\) if \(|X|\) does not embed in
\(|Y|\), whereas for every simplex \(\sigma\) in the second barycentric
subdivision \(X''\), the polyhedron
\(|X''|\setminus st(\sigma)\) embeds in \(|Y|\).  Thus critical complexes are
minimal PL obstructions to embeddability, where the elementary deletion is the
removal of an open star in \(X''\).

Let \(\Gamma(Y)\) denote the set of critical complexes for \(Y\), considered
up to PL homeomorphism.  With this definition, Kuratowski's and Wagner's theorems \cite{K,W} imply that
\[
\Gamma(S^2)=\{K_5,\ K_{3,3},\ S^2\sqcup(\mathrm{pt})\}
\]
see Proposition~\ref{Kuratowski}.  The next natural case is \(Y=S^3\).

The main result of this paper is a complete classification of a natural
graph-theoretic class of critical complexes for \(S^3\).  We consider
complexes of the form
\[
X=(G\times S^1)\cup H,
\]
where \(G\) and \(H\) are graphs and \(H\) is attached along vertices of the
branch set of \(G\times S^1\).  The product part \(G\times S^1\) should be
viewed as a two-dimensional thickening of the graph \(G\), while \(H\) is an
additional one-dimensional graph part.  Collapsing each \(S^1\)-factor gives
the associated reduction graph
\[
\widehat X=G\cup H.
\]

Our classification theorem is the following.

\bigskip
\noindent\textbf{Theorem~\ref{product}.}
\textit{
Let \(X\) be a critical complex for \(S^3\) admitting a decomposition
\[
X=(G\times S^1)\cup H,
\]
where \(G\) and \(H\) are graphs and \(H\) is attached along vertices of the
branch set of \(G\times S^1\).  Then \(|X|\) is homeomorphic to one of the
seven complexes shown in Figure~\ref{seven}.  Conversely, each of the seven
complexes shown in Figure~\ref{seven} is critical for \(S^3\).
}

\begin{figure}[htbp]
	\begin{center}
	\begin{tabular}{cccc}
	\includegraphics[trim=0mm 0mm 0mm 0mm, width=.2\linewidth]{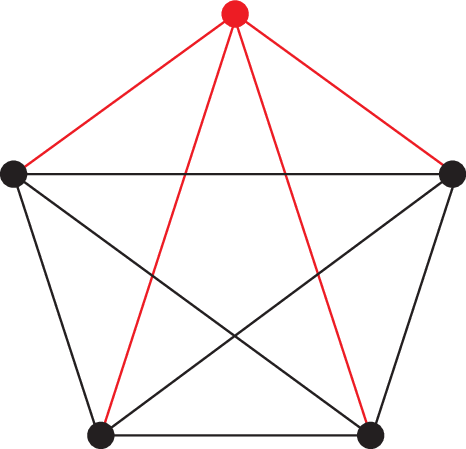}&
	\includegraphics[trim=0mm 0mm 0mm 0mm, width=.2\linewidth]{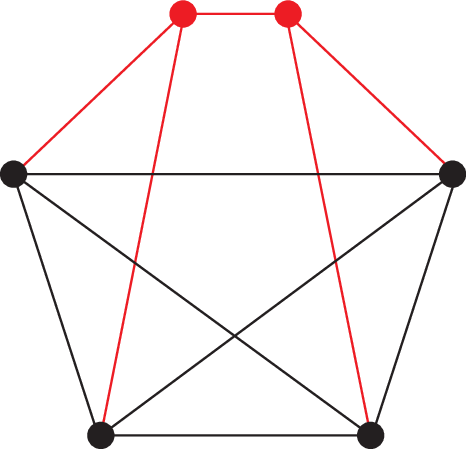}& & \\
    (1) & (2) & & \\
	\includegraphics[trim=0mm 0mm 0mm 0mm, width=.2\linewidth]{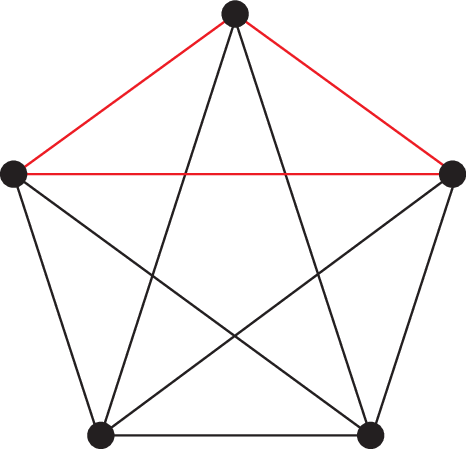}&
	\includegraphics[trim=0mm 0mm 0mm 0mm, width=.2\linewidth]{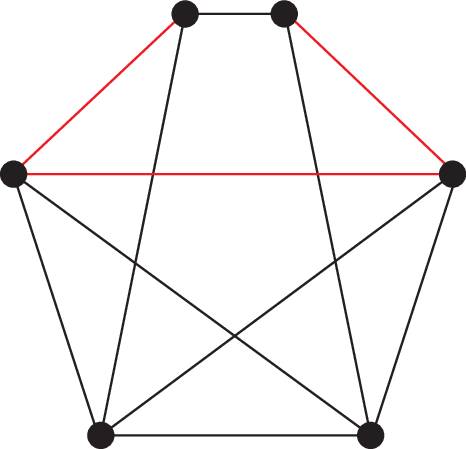}&
	\includegraphics[trim=0mm 0mm 0mm 0mm, width=.2\linewidth]{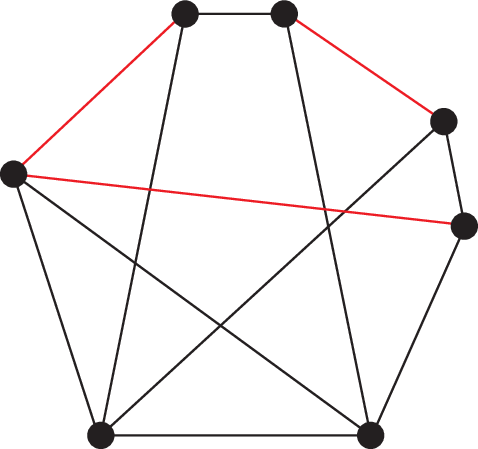}&
	\includegraphics[trim=0mm 0mm 0mm 0mm, width=.2\linewidth]{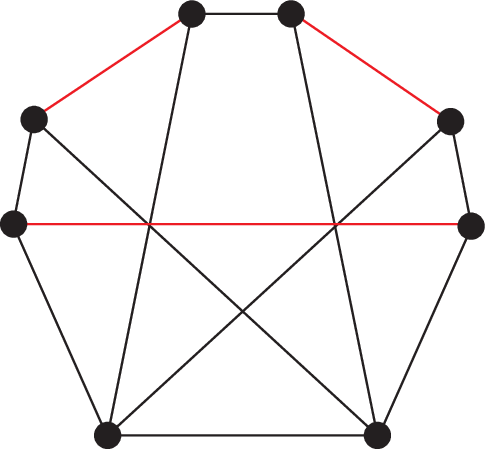}\\
    (3) & (4) & (5) & (6)\\
    \includegraphics[trim=0mm 0mm 0mm 0mm, width=.2\linewidth]{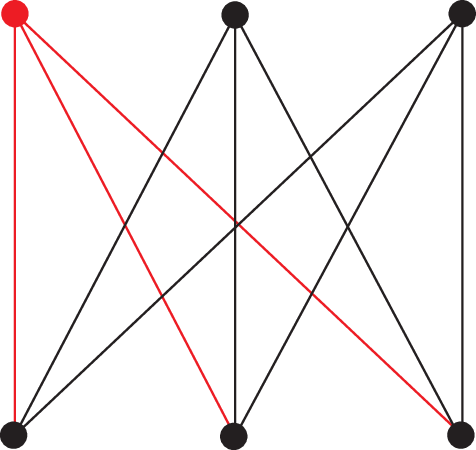}& & & \\
    (7) & & & \\
\end{tabular}
	\end{center}
    \caption{The seven critical complexes of the form \((G\times S^1)\cup H\).
The black graph indicates the reduction of the product part \(G\times S^1\),
and the red graph indicates the graph part \(H\).}
    \label{seven}
\end{figure}

The seven complexes fall into three families.  The first consists of two
\(K_4\)-type complexes.  The second consists of four \(\Theta_4\)-type
complexes.  The third consists of one \(K_{2,3}\)-type complex.  Thus the
classification is finite and explicit, in the same spirit as Kuratowski's
classification of minimal non-planar graphs.

The proof is graph-theoretic in nature.  We first show that if
\(X=(G\times S^1)\cup H\) is critical for \(S^3\), then \(H\) is a forest,
\(\widehat X=G\cup H\) is non-planar, and \(G\) itself is planar.  Hence the
non-embeddability is not caused by the product part alone, but by attaching
the forest \(H\) to the planar graph \(G\).

The second ingredient is a minimality argument.  If a proper subcomplex of
\(X\) is already non-embeddable, then it survives after deleting the open star
of a suitable simplex in \(X''\), contradicting criticality.  This rules out
proper non-embeddable subcomplexes and forces the obstruction to be minimal.

The third ingredient is a finite Kuratowski-type analysis.  Since
\(\widehat X\) is non-planar, its inclusion-minimal non-planar subgraphs are
subdivisions of \(K_5\) or \(K_{3,3}\).  We prove that criticality forces
\(\widehat X\) itself to be inclusion-minimal non-planar.  The remaining
possibilities are then checked by a finite analysis of forest attachments to
\(K_5\) and \(K_{3,3}\).  The embeddability tests are reduced to a planar
face-incidence criterion: each component of the attached forest must have its
attaching vertices on the boundary of a single complementary region.  This
leaves precisely the seven models in Figure~\ref{seven}.

We also prove a general existence result for multibranched surfaces.  Namely,
every regular multibranched surface that is non-embeddable in \(S^3\) contains
a critical subcomplex of the form \(M\cup H\), where \(M\) is a regular
multibranched surface and \(H\) is a graph.  This shows that the class
considered in the main theorem arises naturally in the study of minimal
obstructions to embeddability in \(S^3\).

The paper is organized as follows.
In Section~2 we collect the basic definitions and terminology.
In Section~3 we discuss critical complexes for low-dimensional closed manifolds.
In Section~4 we prove the existence theorem for non-embeddable regular
multibranched surfaces and construct examples coming from
\(K_5\times S^1\) and \(K_{3,3}\times S^1\).
In Section~5 we prove the classification theorem for critical complexes of the
form \((G\times S^1)\cup H\).

\section{Preliminaries}

In this section we collect the basic definitions and terminology used
throughout the paper.  All spaces are considered in the piecewise-linear
category.  A simplicial complex always means a finite simplicial complex, and
\(|X|\) denotes its underlying polyhedron.

\subsection{Critical complexes}

Let \(X\) and \(Y\) be finite simplicial complexes.  We say that \(X\) is
\emph{critical for \(Y\)} if \(|X|\) does not admit an embedding into \(|Y|\),
but for every simplex \(\sigma\) in the second barycentric subdivision \(X''\),
the polyhedron
\[
  |X''|\setminus st(\sigma)
\]
embeds in \(|Y|\).  Here \(st(\sigma)\) denotes the open star of \(\sigma\)
in \(X''\).

Thus a critical complex is a minimal obstruction to embeddability in the
following PL sense: the whole polyhedron does not embed, but every elementary
star-deletion embeds.

Let \(\Gamma(Y)\) denote the set of critical complexes for \(Y\), considered up
to PL homeomorphism.

\begin{remark}\label{rem:criticality}
The use of the second barycentric subdivision is important.  A naive definition
would be to delete a point \(p\in |X|\), but point deletion destroys compactness
and takes us outside the category of polyhedra.  Instead, deleting the open star
of a simplex in \(X''\) gives a polyhedral operation compatible with PL topology.

The second barycentric subdivision is used so that the deleted open star behaves
as a PL regular neighborhood.  Hence the complement
\(|X''|\setminus st(\sigma)\) remains a polyhedron with a well-defined PL
structure.  In this sense, critical complexes are PL analogues of minimal
obstructions, although the notion is not literally a minor-theoretic one.
\end{remark}

\subsection{Two-dimensional and one-dimensional parts}

Let \(X\) be a two-dimensional simplicial complex, and let \(\triangle^i\)
denote the set of \(i\)-simplices of \(X\).  For \(p\in |X|\), let \(N(p;X)\)
denote an open neighborhood of \(p\) in \(|X|\).

The \emph{two-dimensional part} of \(X\) is the union of all two-simplices of
\(X\), denoted by
\[
  X_2=|\triangle^2|.
\]
Define
\[
  S(X)=\{p\in X_2\mid N(p;X)\cong \mathbb{R}^2\},
  \qquad
  B(X)=X_2\setminus S(X).
\]
The connected components of \(S(X)\) are called \emph{sectors}, and the
connected components of \(B(X)\) are called \emph{branches}.  We also set
\[
  \partial X_2
  =
  \{p\in B(X)\mid N(p;X)\cong \mathbb{R}^2_+\}.
\]

The \emph{one-dimensional part} of \(X\) is the underlying polyhedron of the
one-simplices that are not contained in the boundary of any two-simplex:
\[
  X_1=
  \left|
  \{\sigma\in \triangle^1
  \mid
  \sigma\not\subset \partial\tau
  \text{ for every } \tau\in \triangle^2\}
  \right|.
\]
Define
\[
  E(X)=\{p\in X_1\mid N(p;X)\cong \mathbb{R}\},
  \qquad
  V(X)=X_1\setminus E(X).
\]
The connected components of \(E(X)\) are called \emph{edges}, and the connected
components of \(V(X)\) are called \emph{vertices}.

\subsection{Complexes of the form \((G\times S^1)\cup H\)}

Let \(G\) and \(H\) be graphs.  Throughout the paper, the notation \(G\cup H\)
means that \(G\) and \(H\) intersect only in vertices.

The notation
\[
  (G\times S^1)\cup H
\]
denotes a two-dimensional complex whose two-dimensional part is \(G\times S^1\)
and whose one-dimensional part is the graph \(H\).  We assume that
\[
  (G\times S^1)\cap H
\]
is contained in the branch set of \(G\times S^1\).

For such a complex, the \emph{reduction graph} is the graph
\[
  \widehat X=G\cup H
\]
obtained by collapsing each \(S^1\)-factor of \(G\times S^1\) to a point and
leaving \(H\) unchanged.

\subsection{Multibranched surfaces}

A two-dimensional simplicial complex \(X\) is called a
\emph{multibranched surface} if each branch of \(X\) is a simple closed curve
and \(X_1=\emptyset\).

We do not assume that \(\partial X_2=\emptyset\).  Instead, any boundary
component of \(X_2\) is regarded as a branch of degree one, and is therefore
required to be a simple closed curve.

A multibranched surface is called \emph{regular} if every branch has a
well-defined degree and the sector incident to each branch appears with locally
constant wrapping number along that branch.

Multibranched surfaces form a natural class of two-dimensional polyhedra in
embeddability problems in \(S^3\).  They appear naturally in the study of
critical complexes, and they have been studied in several previous works; see
\cite{EMO,MO,O1,O,VS} for background, examples, and embeddability results.

\section{Critical complexes for low-dimensional closed manifolds}

In this section, we determine the critical complexes for several closed PL manifolds of low dimension.

\begin{lemma}\label{dim}
Let $X \in \Gamma(Y)$ for simplicial complexes $X$ and $Y$.
Then $\dim X \le \dim Y$.
\end{lemma}

\begin{proof}
Assume to the contrary that $m:=\dim X > \dim Y=:n$.
Choose an $m$-simplex $\tau$ of $X$, and let $F$ be a facet of $\tau$.
Let $v_F$ be the barycenter of $F$, regarded as a vertex of the second barycentric subdivision $X''$.
Then $\tau''$ contains an $m$-simplex whose interior is disjoint from $st(v_F)$.
Hence $|X''| \setminus st(v_F)$ still contains an $m$-dimensional simplex, and therefore cannot be embedded in the $n$-dimensional polyhedron $|Y|$.
This contradicts the criticality of $X$.
\end{proof}

Let $\bullet$ denote a discrete space consisting of a single point.

\begin{proposition}\label{S^1}
$\Gamma(S^1)=\{S^1 \sqcup \bullet\}$.
\end{proposition}

\begin{proof}
Let $X \in \Gamma(S^1)$.
By Lemma~\ref{dim}, we have $\dim X \le 1$, so $X$ is a finite graph.

Suppose first that $X$ is disconnected.
Since every elementary deletion from $X$ is embeddable in $S^1$, each connected component of $X$ must be a compact $1$-dimensional subpolyhedron of $S^1$.
Hence each component is homeomorphic to a point, an arc, or $S^1$.
Because $X$ itself is not embeddable in $S^1$, one component must be homeomorphic to $S^1$.
Every other component must be a single point:
if one of them contained an edge, then deleting a simplex in that component would still leave a nontrivial component disjoint from a circle, which cannot embed in $S^1$;
and if there were at least two extra points, then deleting one of them would still leave $S^1 \sqcup \bullet$, which is not embeddable in $S^1$.
Therefore $X \cong S^1 \sqcup \bullet$.

Next, suppose that $X$ is connected.
Then $X$ is not homeomorphic to $S^1$, since otherwise it would embed in $S^1$.
Hence $X$ contains a triod, that is, a subspace homeomorphic to a $Y$-shaped graph.
Choose a simplex $\sigma$ in one of the terminal edges of this triod.
Then $|X''| \setminus st(\sigma)$ still contains a triod, and therefore cannot be embedded in $S^1$, contradicting the criticality of $X$.
Thus the connected case does not occur.
\end{proof}

\begin{proposition}\label{Kuratowski}
$\Gamma(S^2)=\{K_5, K_{3,3}, S^2 \sqcup \bullet\}$.
\end{proposition}

\begin{proof}
Let $X \in \Gamma(S^2)$.
By Lemma~\ref{dim}, we have $\dim X \le 2$.

Suppose first that $\dim X=1$.
Then $X$ is a non-planar graph.
By Kuratowski's theorem \cite{K}, $X$ contains a subdivision $\Lambda$ of $K_5$
or $K_{3,3}$.
If $\Lambda$ were a proper subgraph of $X$, then, after passing to the second
barycentric subdivision, we could choose a simplex $\sigma$ of $X''$ with
$st(\sigma)\cap |\Lambda|=\emptyset .$
Hence $|X''|\setminus st(\sigma)$ would still contain $\Lambda$, and therefore
would still be non-planar, contradicting the criticality of $X$.
Thus $\Lambda=X$, and so $X$ is homeomorphic to $K_5$ or $K_{3,3}$.

Now suppose that $\dim X=2$.
Let $A$ be a connected component of $X$ containing a $2$-simplex $\tau$, and let $B$ be the union of the remaining connected components.
Let $b$ be the barycenter of $\tau$ in the first barycentric subdivision of $X$.
By criticality, the polyhedron
$X_b:=|X''| \setminus st(b)$
is embeddable in $S^2$.
Set $A_b:=X_b \cap |A|$.
The boundary of the deleted open star determines a simple closed curve
$C=\partial \overline{st(b)} \subset A_b .$

Let $\varphi \colon X_b \hookrightarrow S^2$ be an embedding.
The curve $\varphi(C)$ bounds a disk $D$ in $S^2$.
Since $C$ is the boundary of the deleted $2$-cell, the only way that
$\varphi(A_b)$ could have no complementary region other than $D$ would be for
$\varphi(A_b)$ itself to be a disk.
Thus, if $A_b$ were not homeomorphic to a disk, then
$S^2 \setminus \varphi(A_b)$ would contain a complementary region different from $D$.
Since $B$ is disjoint from $A_b$, its image $\varphi(B)$ is contained in $S^2 \setminus \varphi(A_b)$. 
By scaling and re-embedding if necessary, we can place $B$ entirely within this other complementary region so that it completely misses $D$.
Then the missing $2$-cell could be restored inside $D$, yielding an embedding of $|X|$ into $S^2$, a contradiction.
Therefore $A_b$ is a disk, and hence $|A|$ is homeomorphic to $S^2$.

Since $X$ itself is not embeddable in $S^2$, the set $B$ is nonempty.
If $B$ contained an edge or at least two points, then deleting a simplex of $B$ would still leave a nonempty component disjoint from $S^2$, so the resulting space would remain non-embeddable in $S^2$.
Thus $B$ is a single point, and therefore
$X \cong S^2 \sqcup \bullet .$
\end{proof}

Let $F_0=S^2$, and for $g>0$ let $F_g$ be the closed orientable surface of genus $g$.
For $g>0$, let $\Omega(F_g)$ denote the set of graphs that are minimal with respect to non-embeddability in $F_g$, where minimality is taken with respect to topological minors.

For background on graph embeddings on surfaces and topological minors, see \cite{MT}.

\begin{theorem}\label{closed surface}
For $g>0$,
$\Gamma(F_g)=\{F_0,\ldots,F_{g-1},\, F_g \sqcup \bullet\}\cup \Omega(F_g).$
\end{theorem}

\begin{proof}
Let $X \in \Gamma(F_g)$.
By Lemma~\ref{dim}, we have $\dim X \le 2$.

Suppose first that $\dim X=1$.
Then $X$ is a graph.
If $X\notin \Omega(F_g)$, then $X$ contains a proper topological minor
$\Gamma$ that is not embeddable in $F_g$.
Equivalently, $X$ contains a proper subgraph $\Lambda$ homeomorphic to a
subdivision of $\Gamma$.
After passing to the second barycentric subdivision, we can choose a simplex
$\sigma$ of $X''$ such that
$st(\sigma)\cap |\Lambda|=\emptyset .$
Then $|X''|\setminus st(\sigma)$ still contains $\Lambda$, and hence is still
not embeddable in $F_g$, contradicting the criticality of $X$.
Therefore $X\in \Omega(F_g)$.

Now suppose that $\dim X=2$.
Let $A$ be a connected component of $X$ containing a $2$-simplex $\tau$, and let $B$ be the union of the remaining connected components.
Let $b$ be the barycenter of $\tau$ in the first barycentric subdivision of $X$.
By criticality,
$X_b:=|X''| \setminus st(b)$
is embeddable in $F_g$.
Set $A_b:=X_b \cap |A|$, and let
$C=\partial \overline{st(b)} \subset A_b .$

Fix an embedding $\varphi \colon X_b \hookrightarrow F_g$.
There are two cases.

\medskip
\noindent
\textit{Case 1.} $\varphi(C)$ bounds a disk $D$ in $F_g \setminus \varphi(A_b)$.

If $A_b$ were not homeomorphic to $F_g$ minus an open disk, then
$F_g \setminus \varphi(A_b)$ would contain a complementary region distinct from $D$;
indeed, otherwise $\varphi(A_b)$ would already be a once-punctured copy of $F_g$.
Since $B$ is disjoint from $A_b$, its image $\varphi(B)$ is contained in $F_g \setminus \varphi(A_b)$.
By scaling and re-embedding if necessary, we can place $B$ entirely within a complementary region distinct from $D$.
Then the missing $2$-cell could be restored inside $D$, yielding an embedding of $|X|$ into $F_g$, a contradiction.
Hence $|A| \cong F_g$.
Since $X$ is not embeddable in $F_g$, the set $B$ must be nonempty; and by the same argument as above, criticality forces $B$ to be a single point.
Therefore
$X \cong F_g \sqcup \bullet .$

\medskip
\noindent
\textit{Case 2.} $\varphi(C)$ does not bound a disk in $F_g \setminus \varphi(A_b)$.

Let $C'$ be a simple closed curve in $F_g$ parallel to $\varphi(C)$ and disjoint from $\varphi(A_b)$.
Cutting $F_g$ along $C'$ and capping the resulting boundary components with disks produces either a connected closed surface $F_h$ with $h<g$, or a disjoint union $F_h \sqcup F_{g-h}$ with $0 \le h < g$.
In either case, the image of $A$ is contained in a closed surface of strictly smaller genus.

If $B \neq \emptyset$, then after placing $B$ away from the image of $A$, and by taking a connected sum with a surface of appropriate genus at a point away from the image of $A$, we can recover an embedding of $|X|$ into $F_g$, contradicting criticality.
Thus $B=\emptyset$, and $X=A$.
Consequently, $|X|$ is homeomorphic to $F_h$ for some $0 \le h < g$.
This proves the theorem.
\end{proof}

Next, we consider critical complexes for a general closed PL manifold.
By Lemma~\ref{dim}, if $X\in \Gamma(M)$ for a closed PL $n$-manifold $M$, then
$\dim X\le n$.
The following theorem describes the extremal case $\dim X=n$.

\begin{theorem}\label{same}
Let $M$ be a connected closed PL $n$-manifold, and let $X\in \Gamma(M)$ be connected.
Then $\dim X=n$ if and only if $|X|$ is a closed PL $n$-manifold and either
\begin{enumerate}
    \item $X\cong S^n$, or
    \item $M\cong X\# M'$ for some closed PL $n$-manifold $M'$ with $M'\not\cong S^n$.
\end{enumerate}
In particular, if $X\not\cong S^n$, then $X$ is homeomorphic to a nontrivial connected summand of $M$.
\end{theorem}

\begin{proof}
The forward implication is immediate, since a closed PL $n$-manifold has
dimension $n$.
We prove the converse.

Assume that $X\in \Gamma(M)$ is connected and $\dim X=n$.
Choose an $n$-simplex $\sigma$ of $X$, and let $b$ be the barycenter of $\sigma$
in the first barycentric subdivision.
Since $X$ is critical, the polyhedron
$P:=|X''|\setminus st(b)$
admits a PL embedding into $M$.
Let
$S=\partial \overline{st(b)}\subset P .$
Then $S$ is an embedded $(n-1)$-sphere in $M$.

We distinguish two cases.

\medskip
\noindent
\textit{Case 1. $S$ separates $M$.}

Let $M_1$ and $M_2$ be the closures of the two components of $M\setminus S$.
Since $P$ is obtained from $|X|$ by deleting an open $n$-ball, the sphere $S$
is the boundary of $P$.
Hence, in a neighborhood of $S$, the image of $P$ lies on only one side of $S$.
Because $P$ is connected, its image is contained in one of the closures of the
two components of $M\setminus S$; denote this side by $M_1$.

If $M_2$ were an $n$-ball, then filling $S$ by this ball would extend the embedding of $P$
to an embedding of $|X|$ into $M$, contradicting $X\in \Gamma(M)$.
Hence $M_2$ is not an $n$-ball.

Let
$\widehat{M}_1:=M_1\cup_S B^n$
be the closed manifold obtained by capping off $\partial M_1=S$.
Then
$M\cong \widehat{M}_1\# \widehat{M}_2,$
where $\widehat{M}_2:=M_2\cup_S B^n$, and $\widehat{M}_2\not\cong S^n$.

We claim that $P=M_1$.
Suppose not.
Choose a point $q\in M_1\setminus P$, and let $B\subset M_1\setminus P$ be a small PL $n$-ball centered at $q$.
Since any two PL $n$-balls in the interior of a connected PL $n$-manifold are ambient isotopic,
the manifold $\widehat{M}_1\setminus \operatorname{int} B$ is PL homeomorphic to $M_1$.
Hence $\widehat{M}_1\setminus \operatorname{int} B$ embeds in $M$.

On the other hand, $|X|$ is obtained from $P$ by attaching one $n$-ball along $S$,
and this attached ball can be taken to be the capping ball in $\widehat{M}_1$.
Since $B$ is disjoint from $P$, it follows that $|X|$ embeds in
$\widehat{M}_1\setminus \operatorname{int} B$, and therefore in $M$.
This is impossible.
Thus $P=M_1$.

Consequently, $|X|$ is homeomorphic to $\widehat{M}_1$, and hence $|X|$ is a closed PL $n$-manifold.
Moreover, either $\widehat{M}_1\cong S^n$, or else $\widehat{M}_1$ is a nontrivial connected summand of $M$.

\medskip
\noindent
\textit{Case 2. $S$ does not separate $M$.}

Choose a simple closed curve $C\subset M$ meeting $S$ transversely in a single point.
Let $N=N(S\cup C)$ be a regular neighborhood of $S\cup C$, and set
$M_1:=\overline{M\setminus N}.$
Then $\partial M_1$ is an $(n-1)$-sphere.

Topologically, $N$ is obtained from $S\times I$ by attaching a $1$-handle
$\alpha\times B^{n-1}$, where $\alpha=C\setminus \operatorname{int}(S\times I)$ is an arc.
The intersection 
$P\cap (\alpha\times B^{n-1})$
consists of finitely many components $\alpha_1\times B^{n-1},\dots,\alpha_k\times B^{n-1}$,
where each $\alpha_i$ is a sub-arc of $\alpha$.
We now modify the embedding of $P$ so as to remove these intersections with the $1$-handle.
For each endpoint $a_i$ of $\alpha_i$, the $(n-2)$-sphere
$\{a_i\}\times \partial B^{n-1}\subset \partial(\alpha_i\times B^{n-1})$
bounds an $(n-1)$-ball in the collar $\partial M_1\times I$.
Hence each piece $\alpha_i\times B^{n-1}$ can be slid into this collar.
Performing this operation for all $i$, we obtain a PL embedding of $P$ into $M_1$.
Equivalently, the original embedding of $P$ can be isotoped off the $1$-handle of $N$.
See Figure~\ref{re-embedding}.

\begin{figure}[htbp]
    \centering
    \includegraphics[width=0.8\textwidth]{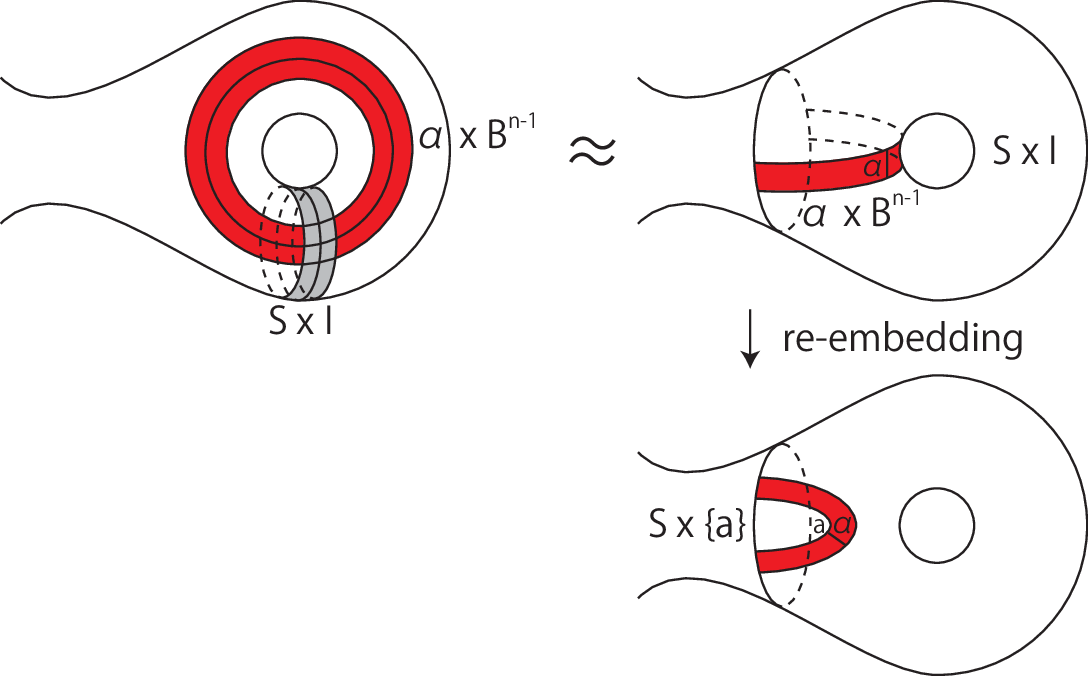}
    \caption{Re-embedding of $|X''|\setminus st(b)$ in $M_1$.}
    \label{re-embedding}
\end{figure}

After this re-embedding, the sphere $S$ becomes parallel to $\partial M_1$.
Therefore, after capping off $\partial M_1$ by an $n$-ball, the sphere $S$ bounds an $n$-ball.
Filling $S$ by this ball recovers an embedding of $|X|$.
Since the capping ball may be realized inside the regular neighborhood $N\subset M$,
this yields an embedding of $|X|$ into $M$, contradicting the criticality of $X$.

Thus Case~2 cannot occur.
The proof is complete.
\end{proof}

\section{Existence of critical subcomplexes in regular multibranched surfaces}

We now restrict attention to regular multibranched surfaces.
Recall that a multibranched surface $X$ is called \emph{regular} if, for each branch
circle $C\subset B(X)$, all boundary components of sectors incident to $C$ map to
$C$ with the same covering degree.
By \cite[Proposition 2.7]{MO}, this condition is equivalent to the existence of an
embedding of $X$ into some closed $3$-manifold.

The next theorem shows that every regular multibranched surface that is
non-embeddable in $S^3$ contains a minimal obstruction whose $2$-dimensional part
is again a regular multibranched surface and whose $1$-dimensional part is a graph.

\begin{theorem}\label{subcomplex}
Let $X$ be a regular multibranched surface.
If $X$ is not embeddable in $S^3$, then $X$ contains a critical subcomplex of the
form $M\cup H$, where $M$ is a regular multibranched surface and $H$ is a graph
(possibly empty).
\end{theorem}

\begin{proof}
Call a subcomplex $Y\subset X$ \emph{admissible} if $|Y|=M\cup H$, where $M$ is a
regular multibranched surface and $H$ is a graph.
Since $X$ itself is admissible and non-embeddable in $S^3$, there exists a
non-embeddable admissible subcomplex.
Among all such subcomplexes, choose one,
$Y=M\cup H \subset X,$
for which the pair
$c(Y):=(s(M),e(H))$
is minimal in the lexicographic order, where $s(M)$ denotes the number of sectors
of $M$ and $e(H)$ the number of edges of $H$.
We claim that $Y$ is critical for $S^3$.

Assume to the contrary that $Y$ is not critical.
Then there exists a simplex $\sigma$ in the second barycentric subdivision $Y''$
such that
$|Y''|\setminus st(\sigma)$
is still non-embeddable in $S^3$.
We consider three cases.

\medskip
\noindent

\textit{Case 1. $st(\sigma)$ is contained in the interior of a sector $S$ of $M$.}

Then $S\setminus st(\sigma)$ is homeomorphic to a regular neighborhood
$N(G\cup \partial S;S)$ of a connected graph $G$ properly embedded in $S$.
Set
$Y_0:=(Y\setminus \operatorname{int} S)\cup G$.

We claim that $Y_0$ is non-embeddable in $S^3$.
Suppose to the contrary that $Y_0$ embeds in $S^3$.
Since $G$ is connected, its image is contained in the closure of a single
complementary region, say $R$, of
$S^3\setminus (Y\setminus \operatorname{int} S)$.
Because $M$ is regular, the boundary components of $S$ meet the incident branch
circles with compatible degree, so a collar neighborhood $N(\partial S;S)$ can be
restored inside $R$ along those branch circles.
Then the band neighborhood $N(G;S)$ can also be realized inside $R$ along the
embedded graph $G$.
Hence
\[
(Y\setminus \operatorname{int} S)\cup N(G\cup \partial S;S)
\cong |Y''|\setminus st(\sigma)
\]
embeds in $S^3$, contradicting the choice of $\sigma$.
Thus $Y_0$ is non-embeddable.

Since $Y_0$ is again admissible and has one fewer sector than $M$, we have
$c(Y_0)<c(Y)$,
contrary to the minimality of $Y$.

\medskip
\noindent
\textit{Case 2. $st(\sigma)$ is contained in the graph part $H$.}

In this case, deleting $st(\sigma)$ simply removes a nonempty subgraph from $H$.
Hence
$Y_0:=M\cup (H\setminus st(\sigma))$
is admissible.
Moreover, $Y_0$ is PL homeomorphic to $|Y''|\setminus st(\sigma)$, and therefore
is non-embeddable in $S^3$.
Moreover, $e(H\setminus st(\sigma))<e(H)$, so again
$c(Y_0)<c(Y),$
a contradiction.

\medskip
\noindent
\textit{Case 3. $st(\sigma)$ is not contained in the interior of any sector of $M$
and is not contained in $H$.}

Then $st(\sigma)$ meets the branch set of $M$ (or the intersection $M\cap H$).
Since $Y$ is $2$-dimensional, $st(\sigma)$ contains a $2$-simplex $\tau$ lying in
the interior of some sector $S$ of $M$.
Because
$|Y''|\setminus st(\sigma)\subset |Y''|\setminus \tau,$
the space $|Y''|\setminus \tau$ is also non-embeddable in $S^3$.
After one further barycentric subdivision, we may choose a simplex $\sigma'$
contained in the interior of $\tau$ such that $st(\sigma')\subset \tau$.
Then
$|Y^{(4)}|\setminus st(\sigma')$
is non-embeddable in $S^3$, and we are reduced to Case~1.
This again contradicts the minimality of $Y$.

All cases lead to contradictions.
Therefore $Y$ is critical for $S^3$.
Since $Y$ is admissible, it has the required form $M\cup H$.
This proves the theorem.
\end{proof}

\subsection{Examples from \(K_5\times S^1\) and \(K_{3,3}\times S^1\)}

Although Theorem~\ref{subcomplex} is not constructive, it suggests natural sources
of critical subcomplexes. In this subsection we introduce two explicit families of
non-embeddable regular multibranched surfaces derived from
$K_5\times S^1$ and $K_{3,3}\times S^1$.

Let $Y_n=K_{1,n}\times S^1$, let $P_n$ denote an $n$-punctured sphere, and let
$D_n$ be a disjoint union of $n$ disks.
Suppose that a multibranched surface $X$ contains a subcomplex $Y_n$.
For integers $i,j\ge 0$ with $i+j=n$, we define a \emph{replacement operation}
by deleting the interiors of the sectors of $Y_n$ together with the branch of order
$n$, and then attaching $\partial P_i$ and $\partial D_j$ to the boundary branches of $Y_n$ by
degree-one maps, one boundary component to each boundary branch.
We assume that the resulting multibranched surface contains no non-orientable
surface as a subcomplex.
This operation preserves the branch degrees.

Applying this operation recursively to $K_5\times S^1$ and $K_{3,3}\times S^1$
produces the following families:
\begin{enumerate}[(1)]
\item $K_5 \times S^1$,
\item $(K_4 \times S^1) \cup P_4$,
\item $(K_4 \times S^1) \cup P_3 \cup D_1$,
\item $(K_4 \times S^1) \cup D_4$,
\item $(K_3 \times S^1) \cup P_3 \cup D_3$,
\item $K_{3,3} \times S^1$,
\item $(K_{2,3} \times S^1) \cup P_3$,
\item $(K_{2,3} \times S^1) \cup D_3$,
\item $(K_{1,3} \times S^1) \cup P_3 \cup D_3$.
\end{enumerate}
We refer to (1)--(5) as the $K_5\times S^1$ family and to (6)--(9) as the
$K_{3,3}\times S^1$ family.

For later use, note that in (5) and (9) the three boundary components of $P_3$
(and also those of $D_3$) are attached to three distinct boundary branches.

\medskip

Let $X$ be a subcomplex of a trivial bundle $F\times S^1$ with projection
$p\colon F\times S^1\to F$.
We say that $X$ is \emph{vertical} if $p^{-1}(p(X))=X$.

\begin{lemma}\label{two_embeddings}
Let $G$ be a connected graph and let
$f\colon G\times S^1 \to S^3$
be an embedding.
Then, after ambient isotopy, one of the following holds:
\begin{enumerate}
\item there exists a knot $K\subset S^3$ and a product structure
      $N(K)\cong D^2\times S^1$ such that $f(G\times S^1)\subset N(K)$ and
      $f(G\times S^1)$ is vertical in $N(K)$;
\item there exists a cable knot $K$ with cabling annulus $A$ such that
      $f(G\times S^1)\subset N(K)\cup N(A)$, the part inside $N(K)$ is vertical,
      and the part inside $N(A)$ is a union of annuli parallel to $A$.
\end{enumerate}
\end{lemma}

\begin{proof}
Let $T$ be a spanning tree of $G$.
Then $N(f(T\times S^1))$ is a solid torus with product structure
$D^2\times S^1$, and $f(T\times S^1)$ is vertical in it.
Let $K=\{0\}\times S^1$ be its core, and write
$E(K)=S^3\setminus \operatorname{int}N(K).$
For each edge $e\in E(G)\setminus E(T)$, the annulus
$A_e=f(e\times S^1)\cap E(K)$
has the property that each boundary component meets a meridian of $N(K)$ exactly
once. By \cite[Lemma 15.26]{B}, each $A_e$ is either boundary-parallel in $E(K)$
or a cabling annulus.
If all such annuli are boundary-parallel, they may be isotoped into $N(K)$, giving
case (1).
Otherwise, one of them is a cabling annulus, and all the others that are not
boundary-parallel are parallel to it; hence we obtain case (2).
\end{proof}

An embedding
$f\colon G\times S^1\to S^3$
is called \emph{standard} if it is of type~(1), the knot $K$ is trivial, and for
some $p\in \partial D^2$, the circle $\{p\}\times S^1$ bounds a meridian disk of
the complementary solid torus.

A \emph{rotation system} for a regular multibranched surface is a cyclic ordering
of the sectors around each branch.
A rotation system on a graph $G$ induces one on $G\times S^1$.

\begin{lemma}\label{standard}
For any connected graph $G$ and embedding
$f\colon G\times S^1\to S^3$,
there exists a standard embedding
$f_0\colon G\times S^1\to S^3$
with the same rotation system as $f$.
\end{lemma}

\begin{proof}
If $f$ is already of type~(1), we re-embed the ambient solid torus as an unknotted
solid torus in $S^3$ and adjust the product structure by Dehn twists so that one
fiber circle becomes a meridian of the complementary solid torus.
This preserves the cyclic ordering of the annuli around each branch.

If $f$ is of type~(2), let $A$ be the cabling annulus.
All annuli of $f(G\times S^1)\cap N(A)$ that are parallel to $A$ are mutually
disjoint, and they lie in a product neighborhood $A\times[-1,1]\subset N(A)$.
Hence we may isotope them simultaneously onto pairwise disjoint level annuli
$A\times\{t_1\},\dots,A\times\{t_m\}$.
For each $k$, let $B_k\subset N(K)$ be a boundary-parallel annulus whose boundary
coincides with that of $A\times\{t_k\}$.
Replacing each level annulus $A\times\{t_k\}$ by $B_k$, we obtain an embedding of
type~(1).
Because the replacements are made independently in pairwise disjoint regions and fix
the boundary circles, they do not change the cyclic ordering of the annuli around
any branch circle.
Therefore the rotation system is unchanged.
\end{proof}

\begin{lemma}\label{planar}
Let $G$ be a connected graph and let
$f\colon G\times S^1\to S^3$
be an embedding.
Then there exists an embedding
$f_0\colon G\to S^2$
with the same rotation system as $f$.
Moreover, the complementary regions of
$S^3\setminus f(G\times S^1)$
correspond bijectively to the complementary regions of
$S^2\setminus f_0(G)$.
\end{lemma}

\begin{proof}
By Lemma~\ref{standard}, we may assume that $f$ is standard.
Hence $f(G\times S^1)$ is vertical in an unknotted solid torus $D^2\times S^1$.
Projecting to $D^2$ yields an embedding of $G$ into $D^2$, hence into $S^2$.
Since the cyclic orderings around the branches determine the adjacency pattern of
the complementary regions, the regions of the two complements correspond naturally.
\end{proof}

\begin{proposition}\label{families_nonembeddable}
None of the members of the $K_5\times S^1$ family and the $K_{3,3}\times S^1$
family embeds in $S^3$.

Moreover, they contain the following subcomplexes:
\begin{enumerate}
\item $K_5 \times S^1 \supset (K_4 \times S^1) \cup K_{1,4}$,
\item $(K_4 \times S^1) \cup P_4 \supset (K_4 \times S^1) \cup K_{1,4}$,
\item $(K_4 \times S^1) \cup P_3 \cup D_1$ (which is critical itself),
\item $(K_4 \times S^1) \cup D_4 \supset (K_{1,3} \times S^1) \cup D_4 \cup K_3$,
\item $(K_3 \times S^1) \cup P_3 \cup D_3$ (which is critical itself),
\item $K_{3,3} \times S^1 \supset (K_{2,3} \times S^1) \cup K_{1,3}$,
\item $(K_{2,3} \times S^1) \cup P_3 \supset (K_{2,3} \times S^1) \cup K_{1,3}$,
\item $(K_{2,3} \times S^1) \cup D_3 \supset (K_{1,3} \times S^1) \cup D_3 \cup K_{1,3}$,
\item $(K_{1,3} \times S^1) \cup P_3 \cup D_3 \supset (K_{1,3} \times S^1) \cup D_3 \cup K_{1,3}$.
\end{enumerate}
\end{proposition}

\begin{proof}
We first prove non-embeddability.

For (1) and (6), if $G\times S^1$ were embeddable in $S^3$, then by
Lemma~\ref{planar} the graph $G$ would embed in $S^2$.
Since $K_5$ and $K_{3,3}$ are non-planar, neither
$K_5\times S^1$ nor $K_{3,3}\times S^1$ embeds in $S^3$.

For (2) and (7), the complement of $K_4\times S^1$ (respectively,
$K_{2,3}\times S^1$) has the same region-incidence pattern as the complement of
a planar embedding of $K_4$ (respectively, $K_{2,3}$) in $S^2$, by
Lemma~\ref{planar}.
Thus each complementary region is incident only to the corresponding boundary
branches.
A sector $P_4$ in case (2), or $P_3$ in case (7), would have to lie entirely in a
single complementary region, but no such region is incident to the required set of
boundary branches.
Hence these complexes are non-embeddable.

For (3), (4), and (8), since one of the sectors attached to
$K_4\times S^1$ or $K_{2,3}\times S^1$ is a disk, we may assume by
Lemma~\ref{standard} that the embedding of the product part is standard.
In a standard embedding, exactly one complementary region has the relevant branch
circles as meridians; all remaining regions meet them as longitudes.
Therefore all disk components must lie in that unique meridional region.
This leaves no region in which the remaining punctured sphere or extra disks can be
attached with the required boundary slopes, giving a contradiction.

For (5), let
$X=(K_3\times S^1)\cup P_3\cup D_3.$
Again the product part may be assumed standard.
Its complement consists of two solid tori, say $R_1$ and $R_2$.
The three disks of $D_3$ must lie in the meridional side, say $R_2$, and therefore
$P_3$ lies in $R_1$.
Then
$\Sigma=P_3\cup D_3$
is an embedded $2$-sphere in $S^3$.
Since every embedded $2$-sphere separates $S^3$, each annulus of
$K_3\times S^1$ would have to lie entirely on one side of $\Sigma$.
But each such annulus has one boundary circle on $P_3$ and one on $D_3$, so this
is impossible.
Hence (5) is non-embeddable.

For (9), the complement of $(K_{1,3}\times S^1)\cup D_3$ has the same
region-incidence structure as the complement of a planar embedding of $K_{1,3}$.
Each complementary region is incident to at most two boundary branches.
Therefore no region can contain a copy of $P_3$ whose three boundary components
attach to three distinct boundary branches.
Hence (9) is non-embeddable.

The listed subcomplex containments are immediate from the construction.\end{proof}

The criticality of the relevant complexes will be established later in
Proposition~\ref{three-types-critical}, where the complementary-region structure
and the effect of elementary deletions are analyzed for the three model types.

\section{Classification of critical complexes of the form \((G\times S^1)\cup H\)}

In this section we classify all critical complexes for $S^3$ whose
$2$-dimensional part is a product $G\times S^1$ and whose $1$-dimensional part is
a graph $H$.

\subsection{Reduction graphs and basic restrictions}

Let
$X=(G\times S^1)\cup H,$
where $G$ and $H$ are graphs and $(G\times S^1)\cap H$ consists only of vertices
in the branch set of $G\times S^1$.
We define the \emph{reduction} of $X$ to be the graph
$\widehat X:=G\cup H,$
obtained by collapsing each circle factor of $G\times S^1$ to a point.
More precisely, regard $S^1$ as $[0,1]/(0\sim 1)$, let
$\pi\colon G\times S^1\to G\times\{0\}\cong G$
be the projection, and extend it by the identity on $H$.
Since $(G\times S^1)\cap H$ consists only of branch vertices, this gives a well-defined map
$f\colon (G\times S^1)\cup H \to G\cup H,$
whose image is $\widehat X$.

In the drawings below, and in the description of the model types later in this
section, an attaching
point which is drawn in the interior of a sector is understood in the following
PL sense: we first subdivide the corresponding edge of the graph $G$ at the
projection of that point and then regard the circle fibre over the new vertex as
a branch circle. Thus the standing convention that the graph part is attached
along vertices of the branch set is preserved up to PL homeomorphism.

The first observation is that the graph part contains no cycle.

\begin{lemma}[$H$ is a forest]\label{forest}
If $X=(G\times S^1)\cup H$ is critical for $S^3$, then $H$ is a forest.
\end{lemma}

\begin{proof}
Assume that some connected component $H'\subset H$ contains a cycle, and choose an
edge $e$ on that cycle.
By the criticality of $X$, the space
$X_e:=X\setminus \operatorname{int} e$
embeds in $S^3$.
Fix such an embedding
$\varphi\colon |X_e|\hookrightarrow S^3.$

Let $N$ be a closed regular neighborhood of $\varphi(G\times S^1)$ in $S^3$.
Since $H$ meets $G\times S^1$ only at vertices in the branch set, after a small
ambient isotopy supported in a neighborhood of $N$ and fixing the attaching
vertices, we may assume that
$\varphi(H'\setminus \operatorname{int} e)\cap \operatorname{int} N=\emptyset .$
Thus $\varphi(H'\setminus \operatorname{int} e)$ is contained in
$S^3\setminus \operatorname{int} N$ and meets $\partial N$ only at the attaching
vertices.

Because $e$ lies on a cycle of $H'$, the graph $H'\setminus \operatorname{int} e$
is still connected.
We claim that its image is contained in the closure of a single connected component
of $S^3\setminus \varphi(G\times S^1)$.
Indeed, if it meets two distinct complementary regions, then a path in
$\varphi(H'\setminus \operatorname{int} e)$ joining points in those two regions
would have to pass from one complementary region to another through
$\varphi(G\times S^1)$.
This is impossible, because
$\varphi(H'\setminus \operatorname{int} e)$ is disjoint from
$\varphi(G\times S^1)$ except at attaching vertices, and near each attaching
vertex the graph locally leaves the branch set through a single side of the
surface part.
Hence there exists a complementary region
$R\subset S^3\setminus \varphi(G\times S^1)$
whose closure contains $\varphi(H'\setminus \operatorname{int} e)$.

The two endpoints of $\varphi(e)$ are attaching vertices, hence they lie in
$\overline{R}\cap \varphi(G\times S^1)$.
Choose points $p,q\in R$ sufficiently close to these two endpoints, respectively,
and join each endpoint to the corresponding point by a short arc in $\overline{R}$
whose interior is contained in $R$.
Since $R$ is connected and
$\varphi(H'\setminus \operatorname{int} e)\cap R$
is $1$-dimensional, general position yields an arc
$\beta\subset R$
joining $p$ and $q$ such that
$\beta\cap \varphi(H'\setminus \operatorname{int} e)=\emptyset $.
Concatenating these three arcs, we obtain an arc
$\alpha\subset \overline{R}$
joining the two endpoints of $\varphi(e)$ such that
$\operatorname{int}\alpha\subset R$
 and 
$\operatorname{int}\alpha\cap \varphi(X_e)=\emptyset $.
Adding $\alpha$ to $\varphi(X_e)$ recovers an embedding of $|X|$ into $S^3$,
contradicting the criticality of $X$.
Therefore $H$ contains no cycle, and hence $H$ is a forest.
\end{proof}

The next lemma explains why the reduction graph must be non-planar.

\begin{lemma}[$\widehat X = G\cup H$ is non-planar]\label{lem:reduction-nonplanar}
If
$X=(G\times S^1)\cup H$
is critical for $S^3$, then its reduction graph
$\widehat X = G\cup H$
is non-planar.
\end{lemma}

\begin{proof}
By Lemma~\ref{forest}, $H$ is a forest.

Assume to the contrary that $\widehat X=G\cup H$ is planar.
Then $\widehat X$ embeds in a disk $D^2\subset S^2$.
Taking the product with $S^1$, we obtain an embedding
\[
G\times S^1 \hookrightarrow D^2\times S^1 \subset S^3.
\]

Because $H$ is a forest, each connected component of $H$ can be embedded in the
closure of a complementary region of the planar embedding of $\widehat X$, with
its attaching vertices fixed.
Therefore the embedding of $G\times S^1$ extends to an embedding of
$X=(G\times S^1)\cup H$
into $S^3$, contradicting the assumption that $X$ is critical for $S^3$.

Hence $\widehat X$ is non-planar.
\end{proof}

\begin{lemma}[Star-deletion minimality lemma]\label{lem:choose-simplex-away}
Let
$X=(G\times S^1)\cup H$
and let
$Y=(G_0\times S^1)\cup H_0\subset X$,
where $G_0\subset G$ and $H_0\subset H$ are subgraphs.

Assume that either $E(G)\setminus E(G_0)\neq\emptyset$ or
$E(H)\setminus E(H_0)\neq\emptyset$.
Then there exists a simplex $\sigma$ of the second barycentric subdivision $X''$
such that
$|Y|\subset |X''|\setminus st(\sigma)$.
In particular, if $|Y|$ does not embed in $S^3$, then
$|X''|\setminus st(\sigma)$ does not embed in $S^3$.
\end{lemma}

\begin{proof}
First suppose that there exists an edge
$e\in E(H)\setminus E(H_0)$.
Choose a simplex $\sigma$ of $X''$ contained in the interior of $e$.
Then $st(\sigma)\subset \operatorname{int} e$, hence $st(\sigma)\cap |Y|=\emptyset$.
Therefore
$|Y|\subset |X''|\setminus st(\sigma)$.

Next suppose that there exists an edge
$e\in E(G)\setminus E(G_0)$.
Then the sector $e\times S^1$ is an annulus contained in the $2$-dimensional part of
$X$ and disjoint from $|G_0\times S^1|$ except possibly along its boundary branch
circles. Choose a $2$-simplex $\tau$ in the interior of $e\times S^1$, and let
$\sigma$ be the barycenter of $\tau$, regarded as a vertex of $X''$.
Then $st(\sigma)$ is contained in the interior of $e\times S^1$, so again
$st(\sigma)\cap |Y|=\emptyset$.
Hence
$|Y|\subset |X''|\setminus st(\sigma)$.

The final assertion is immediate.
\end{proof}

Lemmas~\ref{lem:reduction-nonplanar} and~\ref{lem:choose-simplex-away}
will be used repeatedly below.
In particular, Lemma~\ref{lem:reduction-nonplanar} shows that if
$X=(G\times S^1)\cup H$ is critical for $S^3$, then its reduction graph
$\widehat X$ is non-planar.
We shall use this fact repeatedly below.

\begin{lemma}[minimal non-planar graph is $K_5$ or $K_{3,3}$-type]\label{kuratowski-core}
Let $\Gamma$ be a non-planar graph such that $\Gamma$ has no proper non-planar
subgraph.
Then $\Gamma$ is a subdivision of either $K_5$ or $K_{3,3}$.
\end{lemma}

\begin{proof}
By Kuratowski's theorem, $\Gamma$ contains a subdivision of $K_5$ or $K_{3,3}$.
Such a subdivision is non-planar.
Since $\Gamma$ has no proper non-planar subgraph, it must be all of $\Gamma$.
\end{proof}

\begin{figure}[htbp]
	\begin{center}
	\begin{tabular}{ccc}
	\includegraphics[trim=0mm 0mm 0mm 0mm, width=.22\linewidth]{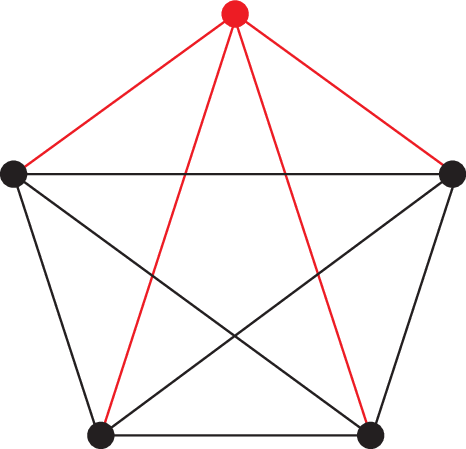}&
	\includegraphics[trim=0mm 0mm 0mm 0mm, width=.22\linewidth]{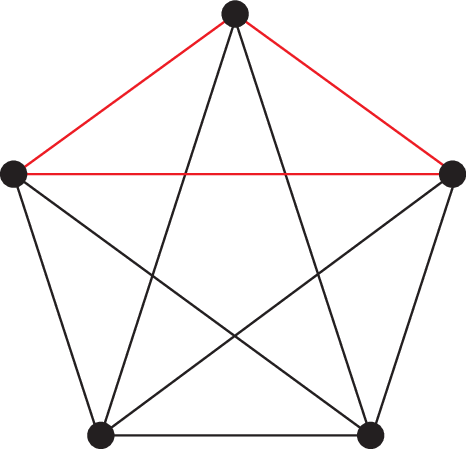}&
	\includegraphics[trim=0mm 0mm 0mm 0mm, width=.22\linewidth]{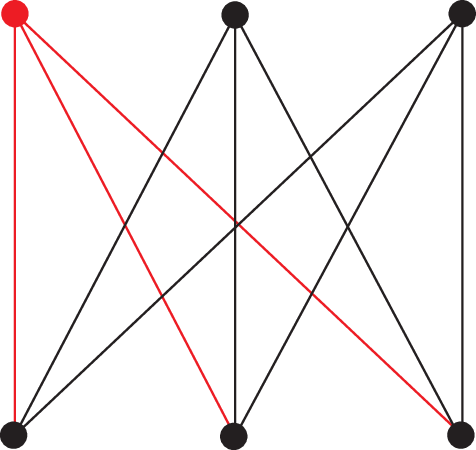}\\
	(i) & (ii) & (iii)
	\end{tabular}
	\end{center}
\caption{Three contracted models of the seven reduction graphs. More precisely,
the reductions corresponding to Figure~\ref{seven}(1),(2) contract to (i),
those corresponding to Figure~\ref{seven}(3)--(6) contract to (ii), and that
corresponding to Figure~\ref{seven}(7) contracts to (iii).}
	\label{3types}
\end{figure}

\begin{lemma}[$G$ is planar]\label{lem:G-planar}
If $X=(G\times S^1)\cup H$ is critical for $S^3$, then $G$ is planar.
\end{lemma}

\begin{proof}
Assume, to the contrary, that $G$ is non-planar.  Choose an
inclusion-minimal non-planar subgraph $G_0\subset G$.  By Kuratowski's
theorem, after suppressing degree-two vertices, $G_0$ is either $K_5$ or
$K_{3,3}$.

Consider the subcomplex
  $X_0:=G_0\times S^1\subset X $.
We first note that $X_0$ is not embeddable in $S^3$.  Indeed, if
$G_0\times S^1$ embedded in $S^3$, then Lemma~\ref{planar} would
imply that $G_0$ embeds in $S^2$, contradicting the non-planarity of
$G_0$.

If either $G_0\subsetneq G$ or $H\ne\emptyset$, then
Lemma~\ref{lem:choose-simplex-away} gives a simplex $\sigma$ of $X''$ such
that
  $|X_0|\subset |X''|\setminus st(\sigma)$.
Thus $|X''|\setminus st(\sigma)$ is non-embeddable in $S^3$, contradicting
criticality.  Hence we may assume
  $X=X_0=G_0\times S^1 $.

It remains to exclude this last possibility.  Suppose first that the
Kuratowski core of $G_0$ is $K_5$.  The product $K_5\times S^1$ contains a
proper subcomplex homeomorphic to
  $(K_4\times S^1)\cup K_{1,4}$,
obtained by deleting one vertex of the $K_5$-core from the product part and
retaining, inside the four incident annuli, four arcs from the deleted
branch circle to the four remaining branch circles.  By
Proposition~\ref{families_nonembeddable}, this subcomplex is
non-embeddable in $S^3$.  If $G_0$ is a subdivision of $K_5$, the same
construction is performed after subdividing the corresponding annuli; the
resulting subcomplex is PL-homeomorphic to the same model.

Similarly, if the Kuratowski core of $G_0$ is $K_{3,3}$, then
$G_0\times S^1$ contains a proper subcomplex homeomorphic to
  $(K_{2,3}\times S^1)\cup K_{1,3}$,
again obtained by replacing the three annuli incident to one deleted vertex
by three arcs.  This subcomplex is non-embeddable by
Proposition~\ref{families_nonembeddable}.

In either case $X$ contains a proper non-embeddable subcomplex $Y$.  After
subdividing $X$ if necessary, we may regard $Y$ as a subcomplex.  Since
$Y\subsetneq X$, there is a simplex $\sigma$ of $X''$ with
$st(\sigma)\cap |Y|=\emptyset$.  Hence
  $|Y|\subset |X''|\setminus st(\sigma)$,
and $|X''|\setminus st(\sigma)$ is non-embeddable in $S^3$, again
contradicting criticality.  Therefore $G$ must be planar.
\end{proof}

\subsection{Region-incidence criteria}

\begin{lemma}[Single-region criterion]\label{lem:single-region-criterion}
Let $P$ be an embedded product part of the form $G_0\times S^1$ in $S^3$, with a
fixed rotation system, and let $T$ be a forest attached to $P$ only at the allowed
attaching vertices. Suppose that every attaching vertex of $T$ is a terminal
vertex of the corresponding component of $T$. Then $P\cup T$ embeds in $S^3$,
extending the given embedding of $P$, if and only if each connected component
$T_i$ of $T$ can be assigned to a complementary region $R_i$ of
$S^3\setminus P$ such that $T_i\cap P\subset \overline{R_i}$.
\end{lemma}

\begin{proof}
If such regions $R_i$ are given, then each tree $T_i$ may be embedded in the
$3$-manifold $\overline{R_i}$ with its terminal vertices fixed on $T_i\cap P$.
Since the components of $T$ are disjoint trees, these embeddings may be chosen
pairwise disjoint by general position. This gives an embedding of $P\cup T$.

Conversely, suppose $P\cup T$ is embedded. For a component $T_i$ of $T$, the
open graph $T_i\setminus (T_i\cap P)$ is connected and is disjoint from $P$.
Hence it is contained in a single complementary region $R_i$ of $S^3\setminus P$.
Taking the closure, all attaching vertices of $T_i$ lie in $\overline{R_i}$.
\end{proof}

\begin{lemma}[Region-incidence criterion after puncturing]\label{lem:region-incidence-puncture}
Let $P=G_0\times S^1$ be a product part embedded in $S^3$, and fix its rotation
system. Let $\Gamma\subset S^2$ be the planar graph with the same rotation system,
as obtained from Lemma~\ref{planar}. Let $T$ be a forest attached to $P$ at terminal
vertices of $T$.
Then $P\cup T$ embeds in $S^3$, extending the embedding of $P$, if and only if for
each component $T_i$ of $T$, the set of attaching vertices of $T_i$ is contained in
the boundary of one complementary region of $S^2\setminus \Gamma$.
Moreover, if a small open disk is deleted from the sector corresponding to an edge
$e$ of $\Gamma$, then the same criterion holds after replacing the two faces of
$S^2\setminus \Gamma$ adjacent to $e$ by their union.
\end{lemma}

\begin{proof}
The first assertion is Lemma~\ref{lem:single-region-criterion} expressed after
projecting a standard embedding of $G_0\times S^1$ to the disk. The complementary
regions of $S^3\setminus (G_0\times S^1)$ correspond bijectively to the complementary
regions of $S^2\setminus \Gamma$ by Lemma~\ref{planar}. A tree component of $T$ can be
placed in a complementary region precisely when all of its attaching vertices lie
on the boundary of that region.

For the second assertion, deleting a small disk from the sector corresponding to
$e$ creates a passage between the two complementary regions adjacent to that
sector. In the projected planar picture this is exactly the operation of merging
the two faces adjacent to $e$. The same single-region criterion then applies to the
modified list of complementary regions.
\end{proof}

\subsection{Finite Kuratowski-type checks}

We use the notation of Figure~\ref{fig:nine-forests}: 
$L_m$ denotes the path on $m$ vertices, and $T_5$ denotes the tree on
five vertices with degree sequence $(3,2,1,1,1)$ which is not $L_5$.

\begin{figure}[htbp]
\centering
\begin{tikzpicture}[
    scale=1,
    vertex/.style={circle, fill=black, inner sep=1.4pt},
    lbl/.style={font=\small}
]


\begin{scope}[shift={(0,0)}]
    \node[vertex] (a) at (0,0) {};
    \node[vertex] (b) at (1.2,0) {};
    \draw (a)--(b);
    \node[lbl] at (0.6,-0.7) {$K_2$};
\end{scope}

\begin{scope}[shift={(3.2,0)}]
    \node[vertex] (a) at (0,0) {};
    \node[vertex] (b) at (0.8,0) {};
    \node[vertex] (c) at (1.6,0) {};
    \draw (a)--(b)--(c);
    \node[lbl] at (0.8,-0.7) {$L_3$};
\end{scope}

\begin{scope}[shift={(7.0,0)}]
    \node[vertex] (a1) at (0,0) {};
    \node[vertex] (a2) at (0.9,0) {};
    \node[vertex] (b1) at (2.0,0) {};
    \node[vertex] (b2) at (2.9,0) {};
    \draw (a1)--(a2);
    \draw (b1)--(b2);
    \node[lbl] at (1.45,-0.7) {$2K_2$};
\end{scope}


\begin{scope}[shift={(0,-2.4)}]
    \node[vertex] (c) at (0.8,0) {};
    \node[vertex] (u) at (0.8,0.8) {};
    \node[vertex] (l) at (0,0) {};
    \node[vertex] (r) at (1.6,0) {};
    \draw (c)--(u);
    \draw (c)--(l);
    \draw (c)--(r);
    \node[lbl] at (0.8,-0.8) {$K_{1,3}$};
\end{scope}

\begin{scope}[shift={(3.2,-2.4)}]
    \node[vertex] (a) at (0,0) {};
    \node[vertex] (b) at (0.8,0) {};
    \node[vertex] (c) at (1.6,0) {};
    \node[vertex] (d) at (2.4,0) {};
    \draw (a)--(b)--(c)--(d);
    \node[lbl] at (1.2,-0.7) {$L_4$};
\end{scope}

\begin{scope}[shift={(7.0,-2.4)}]
    \node[vertex] (a) at (0,0) {};
    \node[vertex] (b) at (0.8,0) {};
    \node[vertex] (c) at (1.6,0) {};
    \draw (a)--(b)--(c);

    \node[vertex] (d) at (2.8,0) {};
    \node[vertex] (e) at (3.7,0) {};
    \draw (d)--(e);

    \node[lbl] at (1.85,-0.7) {$L_3\sqcup K_2$};
\end{scope}


\begin{scope}[shift={(0,-4.9)}]
    \node[vertex] (c) at (0.8,0) {};
    \node[vertex] (u) at (0.8,0.9) {};
    \node[vertex] (d) at (0.8,-0.9) {};
    \node[vertex] (l) at (0,0) {};
    \node[vertex] (r) at (1.6,0) {};
    \draw (c)--(u);
    \draw (c)--(d);
    \draw (c)--(l);
    \draw (c)--(r);
    \node[lbl] at (0.8,-1.4) {$K_{1,4}$};
\end{scope}

\begin{scope}[shift={(3.2,-4.9)}]
    \node[vertex] (v1) at (0,0) {};
    \node[vertex] (v2) at (0.8,0) {};
    \node[vertex] (v3) at (1.6,0) {};
    \node[vertex] (v4) at (0.8,0.8) {};
    \node[vertex] (v5) at (1.6,-0.8) {};
    \draw (v1)--(v2)--(v3);
    \draw (v2)--(v4);
    \draw (v3)--(v5);
    \node[lbl] at (0.8,-1.4) {$T_5$};
\end{scope}

\begin{scope}[shift={(7.0,-4.9)}]
    \node[vertex] (a) at (0,0) {};
    \node[vertex] (b) at (0.8,0) {};
    \node[vertex] (c) at (1.6,0) {};
    \node[vertex] (d) at (2.4,0) {};
    \node[vertex] (e) at (3.2,0) {};
    \draw (a)--(b)--(c)--(d)--(e);
    \node[lbl] at (1.6,-0.7) {$L_5$};
\end{scope}

\end{tikzpicture}
\caption{The nine nonempty forest types in $K_5$ up to automorphism.}
\label{fig:nine-forests}
\end{figure}

\begin{lemma}[Finite $K_5$ check]\label{lem:K5-case-check}
Let $\overline H$ be a nonempty forest in $K_5$, and put
$\overline G:=K_5-E(\overline H)$.
Up to automorphism of $K_5$, the following hold.

\begin{enumerate}
\item[(1)]
If $\overline H\cong 2K_2$, $\overline H\cong L_3\sqcup K_2$,
$\overline H\cong T_5$, or $\overline H\cong L_5$, then
$(\overline G\times S^1)\cup \overline H$ embeds in $S^3$.

\item[(2)]
If $\overline H\cong K_2$, $\overline H\cong L_3$, or
$\overline H\cong K_{1,3}$, then
$(\overline G\times S^1)\cup \overline H$ is non-embeddable in $S^3$,
but it is not critical.

\item[(3)]
If $\overline H\cong K_{1,4}$ or $\overline H\cong L_4$, then the
corresponding reduction graph is one of the six graphs
Figure~\ref{seven}(1)--Figure~\ref{seven}(6).
\end{enumerate}
\end{lemma}

\begin{proof}
Label the vertices of $K_5$ by $\{1,2,3,4,5\}$.

\smallskip
\noindent
\textit{(1) Embeddable cases.}
We choose representatives for $\overline H$ and give planar drawings of
$\overline G$ for which the components of $\overline H$ can be placed in the
corresponding complementary face-balls. By the face-placement criterion for
complexes of the form $(\overline G\times S^1)\cup\overline H$, this gives
embeddings in $S^3$.

\smallskip
\noindent
\underline{$\overline H\cong 2K_2$.}
Take $E(\overline H)=\{12,34\}$. Then
$\overline G=K_5-\{12,34\}$ admits a planar embedding with outer cycle
$1-3-2-4-1$, and with the vertex $5$ joined to each of $1,2,3,4$ in the
interior. In this drawing, the pairs $\{1,2\}$ and $\{3,4\}$ lie on the
boundary of the outer face. Hence the two components of $\overline H$ can be
placed disjointly in the corresponding outer face-ball.

\smallskip
\noindent
\underline{$\overline H\cong L_3\sqcup K_2$.}
Take $E(\overline H)=\{12,23,45\}$. Then
$E(\overline G)=\{13,14,15,24,25,34,35\}$. There is a planar drawing of
$\overline G$ in which the edge $12$ can be realized in one complementary
face-ball, the edge $23$ in an adjacent complementary face-ball meeting the
first one along the branch component corresponding to $2$, and the edge
$45$ in a complementary face-ball disjoint from the path component. Thus the
tree component $1-2-3$ and the edge component $45$ of $\overline H$ are
realized without mutual intersections. Hence
$(\overline G\times S^1)\cup\overline H$ embeds in $S^3$.

\smallskip
\noindent
\underline{$\overline H\cong T_5$.}
Take $E(\overline H)=\{13,23,24,35\}$. Then
$E(\overline G)=\{12,14,15,25,34,45\}$. Thus $\overline G$ is the $4$-cycle
$1-2-5-4-1$, together with the diagonal $15$ and the extra edge $34$ attached
at $4$. Draw the $4$-cycle in the plane, draw the diagonal $15$ inside it, and
draw the edge $34$ outside it. Then all five vertices occur on the boundary of
the outer complementary region, whose boundary walk is $1-2-5-4-3-4-1$.
Since $\overline H$ is a tree with all its vertices on this boundary, it can be
placed in the corresponding outer face-ball. Hence
$(\overline G\times S^1)\cup\overline H$ embeds in $S^3$.

\smallskip
\noindent
\underline{$\overline H\cong L_5$.}
Take $E(\overline H)=\{14,24,25,35\}$. Then $\overline G$ is the $5$-cycle
$1-2-3-4-5-1$, together with the diagonal $13$. All five vertices lie on the
boundary of the outer face. Therefore the path $\overline H$ can be placed in
the corresponding outer face-ball, and
$(\overline G\times S^1)\cup\overline H$ embeds in $S^3$.

This proves the embeddability assertions in (1).

\smallskip
\noindent
\textit{(2) Non-embeddable but non-critical cases.}
We treat the three cases $\overline H\cong K_2$, $\overline H\cong L_3$, and
$\overline H\cong K_{1,3}$ uniformly. In each case, after subdividing the
relevant sectors if necessary, the complex contains a proper subcomplex of
$K_4$-type, namely a subcomplex homeomorphic to
$(K_4\times S^1)\cup K_{1,4}$. By the preceding $K_4$-type obstruction, this
subcomplex is non-embeddable in $S^3$. Hence the original complex is
non-embeddable.

It remains to see that the original complex is not critical.

\smallskip
\noindent
\underline{$\overline H\cong K_2$.}
Take $E(\overline H)=\{12\}$. Use the product part on the vertices
$\{2,3,4,5\}$ as $K_4\times S^1$. The edge $12$ of $\overline H$, together
with core arcs chosen in the sectors $13\times S^1$, $14\times S^1$, and
$15\times S^1$, gives a $K_{1,4}$ graph with centre at $1$ and leaves on the
four branch components corresponding to $2,3,4,5$. Thus the original complex
contains a proper non-embeddable $K_4$-type subcomplex.

\smallskip
\noindent
\underline{$\overline H\cong L_3$.}
Take $E(\overline H)=\{12,23\}$. Use the product part on the vertices
$\{1,3,4,5\}$ as $K_4\times S^1$. The two edges $21$ and $23$ of
$\overline H$, together with core arcs in the sectors $24\times S^1$ and
$25\times S^1$, give a $K_{1,4}$ graph with centre at $2$ and leaves
corresponding to $1,3,4,5$. Hence again the original complex contains a proper
non-embeddable $K_4$-type subcomplex.

\smallskip
\noindent
\underline{$\overline H\cong K_{1,3}$.}
Take $E(\overline H)=\{12,13,14\}$. Then
$E(\overline G)=\{15,23,24,25,34,35,45\}$. Use the product part on the
vertices $\{2,3,4,5\}$ as $K_4\times S^1$. The three edges $12,13,14$ of
$\overline H$, together with a core arc in the sector $15\times S^1$, give a
$K_{1,4}$ graph with centre at $1$ and leaves corresponding to $2,3,4,5$.
Thus this case also contains a proper non-embeddable $K_4$-type subcomplex.

In all three cases, choose a simplex $\sigma$ of the second barycentric
subdivision in a small two-dimensional part of one of the sectors not used by
the chosen $K_4$-type subcomplex. Then the open star $st(\sigma)$ is disjoint
from that subcomplex. Consequently $|X''|\setminus st(\sigma)$ still contains
a non-embeddable subcomplex. Therefore
$(\overline G\times S^1)\cup\overline H$ is not critical.

\smallskip
\noindent
\textit{(3) The remaining two cases.}
If $\overline H\cong K_{1,4}$, then $\overline G\cong K_4$, and undoing the
suppressed degree-two vertices gives exactly the two graphs
Figure~\ref{seven}(1) and Figure~\ref{seven}(2).

If $\overline H\cong L_4$, then undoing the suppressed degree-two vertices
gives exactly the four graphs Figure~\ref{seven}(3)--Figure~\ref{seven}(6).

This proves the lemma.
\end{proof}

\begin{lemma}[Finite $K_{3,3}$ check]\label{lem:K33-case-check}
Let the two parts of $K_{3,3}$ be
$A=\{a_1,a_2,a_3\}$ and $B=\{b_1,b_2,b_3\}$, and write $a_i b_j$
for the edge joining $a_i$ to $b_j$.  Let $\overline H$ be a nonempty
forest in $K_{3,3}$, and put $\overline G:=K_{3,3}-E(\overline H)$.
Up to automorphism of $K_{3,3}$, the following hold.

\begin{enumerate}
\item[(1)]
If $\overline H\cong 2K_2$, $\overline H\cong L_4$,
$\overline H\cong L_3\sqcup K_2$, $\overline H\cong 3K_2$,
$\overline H\cong L_5$, $\overline H\cong L_4\sqcup K_2$, or
$\overline H\cong L_6$, then
$(\overline G\times S^1)\cup\overline H$ embeds in $S^3$.

\item[(2)]
If $\overline H\cong K_2$, $\overline H\cong L_3$, or
$\overline H\cong L_3\sqcup L_3$, then
$(\overline G\times S^1)\cup\overline H$ is non-embeddable in $S^3$,
but it is not critical.

\item[(3)]
If $\overline H$ properly contains a $K_{1,3}$ subforest, then
$(\overline G\times S^1)\cup\overline H$ contains a proper
non-embeddable subcomplex, and hence is not critical.

\item[(4)]
If $\overline H\cong K_{1,3}$, then the corresponding reduction graph is
the $K_{2,3}$-type graph shown in Figure~\ref{seven}(7).
\end{enumerate}
\end{lemma}

\begin{proof}
We first enumerate the possible forests.  The automorphism group of
$K_{3,3}$ is generated by permutations of $A$, permutations of $B$, and the
interchange of the two parts.  Hence a nonempty forest is determined, up to
automorphism, by its degree sequence together with its incidence with the
bipartition.  The following representatives exhaust all possibilities.

\begin{center}
\begin{tabular}{c|l}
type & representative for $E(\overline H)$ \\
\hline
$K_2$ & $\{a_1b_1\}$ \\
$L_3$ & $\{a_1b_1,a_1b_2\}$ \\
$2K_2$ & $\{a_1b_1,a_2b_2\}$ \\
$L_4$ & $\{a_1b_1,a_1b_2,a_2b_1\}$ \\
$L_3\sqcup K_2$ & $\{a_1b_1,a_1b_2,a_2b_3\}$ \\
$3K_2$ & $\{a_1b_1,a_2b_2,a_3b_3\}$ \\
$L_5$ & $\{a_1b_1,a_1b_2,a_2b_1,a_2b_3\}$ \\
$L_4\sqcup K_2$ & $\{a_1b_1,a_1b_2,a_2b_1,a_3b_3\}$ \\
$L_3\sqcup L_3$ & $\{a_1b_1,a_1b_2,a_2b_3,a_3b_3\}$ \\
$K_{1,3}$ & $\{a_1b_1,a_1b_2,a_1b_3\}$ \\
$(3,2,1,1,1)$ & $\{a_1b_1,a_1b_2,a_1b_3,a_2b_1\}$ \\
$(3,3,1,1,1,1)$ & $\{a_1b_1,a_1b_2,a_1b_3,a_2b_1,a_3b_1\}$ \\
$(3,2,2,1,1,1)$ & $\{a_1b_1,a_1b_2,a_1b_3,a_2b_1,a_3b_2\}$ \\
$L_6$ & $\{a_1b_1,a_1b_2,a_2b_1,a_2b_3,a_3b_2\}$
\end{tabular}
\end{center}

Here and below, when $\overline G$ is disconnected, ``cofacial'' means lying
on the same boundary component of a complementary region, as required by
Lemma~\ref{lem:region-incidence-puncture}.

\smallskip
\noindent
\textit{(1) Embeddable cases.}
For the seven representatives
$2K_2$, $L_4$, $L_3\sqcup K_2$, $3K_2$, $L_5$,
$L_4\sqcup K_2$, and $L_6$, one checks directly from planar drawings of
$\overline G$ that each component of $\overline H$ has all its attaching
vertices on one boundary component of a complementary region of $\overline G$.
Therefore, by Lemma~\ref{lem:region-incidence-puncture}, the complex
$(\overline G\times S^1)\cup\overline H$ embeds in $S^3$.

For clarity, the representatives used in this finite check are the following:
for $2K_2$, use $\{a_1b_1,a_2b_2\}$; for $L_4$, use
$\{a_1b_1,a_1b_2,a_2b_1\}$; for $L_3\sqcup K_2$, use
$\{a_1b_1,a_1b_2,a_2b_3\}$; for $3K_2$, use
$\{a_1b_1,a_2b_2,a_3b_3\}$; for $L_5$, use
$\{a_1b_1,a_1b_2,a_2b_1,a_2b_3\}$; for $L_4\sqcup K_2$, use
$\{a_1b_1,a_1b_2,a_2b_1,a_3b_3\}$; and for $L_6$, use
$\{a_1b_1,a_1b_2,a_2b_1,a_2b_3,a_3b_2\}$.

\smallskip
\noindent
\textit{(2) Non-embeddable but non-critical cases.}
Consider first $\overline H\cong K_2$.  Take
$E(\overline H)=\{a_1b_1\}$.  If the corresponding complex embedded in
$S^3$, then by Lemma~\ref{lem:region-incidence-puncture} the vertices
$a_1$ and $b_1$ would be cofacial in a planar embedding of
$K_{3,3}-a_1b_1$.  The missing edge $a_1b_1$ could then be inserted in that
region, giving a planar embedding of $K_{3,3}$, a contradiction.  Hence the
complex is non-embeddable.

The same argument applies to $\overline H\cong L_3$.  Taking
$E(\overline H)=\{a_1b_1,a_1b_2\}$, an embedding would force
$a_1,b_1,b_2$ to be cofacial in a planar embedding of
$K_{3,3}-\{a_1b_1,a_1b_2\}$.  Then the path $b_1-a_1-b_2$ could be inserted
in that region, again producing a planar embedding of $K_{3,3}$.

For $\overline H\cong L_3\sqcup L_3$, take
$E(\overline H)=\{a_1b_1,a_1b_2,a_2b_3,a_3b_3\}$.  If the complex embedded,
then the two attaching triples $a_1,b_1,b_2$ and $b_3,a_2,a_3$ would be
realizable in the complementary regions of a planar embedding of
$\overline G$.  The two path components of $\overline H$ could then be
inserted, giving a planar realization of all edges of $K_{3,3}$, again
impossible.  Thus this case is also non-embeddable.

It remains to see that these three cases are not critical.  In each case,
there is a sector whose puncture leaves the same region-incidence obstruction.
A direct finite check gives the following choices.

\begin{center}
\begin{tabular}{c|c|c}
$\overline H$ & representative for $E(\overline H)$ & punctured sector \\
\hline
$K_2$ & $\{a_1b_1\}$ & $a_2b_1$ \\
$L_3$ & $\{a_1b_1,a_1b_2\}$ & $a_2b_2$ \\
$L_3\sqcup L_3$ & $\{a_1b_1,a_1b_2,a_2b_3,a_3b_3\}$ & $a_2b_1$
\end{tabular}
\end{center}

After puncturing the indicated sector, the relevant attaching vertices are
still not cofacial in the sense of Lemma~\ref{lem:region-incidence-puncture}.
Thus the punctured complex remains non-embeddable.  Hence in each of the
three cases there is a simplex $\sigma$ in the second barycentric subdivision
such that $|X''|\setminus st(\sigma)$ is non-embeddable.  Therefore these
complexes are non-embeddable but not critical.

\smallskip
\noindent
\textit{(3) Forests properly containing a $K_{1,3}$ subforest.}
There are three remaining representatives which properly contain the
$K_{1,3}$ subforest $\{a_1b_1,a_1b_2,a_1b_3\}$, namely
$\{a_1b_1,a_1b_2,a_1b_3,a_2b_1\}$,
$\{a_1b_1,a_1b_2,a_1b_3,a_2b_1,a_3b_1\}$, and
$\{a_1b_1,a_1b_2,a_1b_3,a_2b_1,a_3b_2\}$.
The subforest $\{a_1b_1,a_1b_2,a_1b_3\}$ gives the $K_{2,3}$-type
obstruction.  Hence each corresponding complex contains a proper
non-embeddable subcomplex.  Choosing a simplex $\sigma$ whose open star is
disjoint from this proper subcomplex, the space $|X''|\setminus st(\sigma)$
still contains a non-embeddable subcomplex.  Therefore these cases are not
critical.

\smallskip
\noindent
\textit{(4) The $K_{1,3}$ case.}
The only remaining pattern is $\overline H\cong K_{1,3}$.  Taking
$E(\overline H)=\{a_1b_1,a_1b_2,a_1b_3\}$, we have
$\overline G=K_{3,3}-\{a_1b_1,a_1b_2,a_1b_3\}\cong K_{2,3}$.  This is
exactly the reduction of the $K_{2,3}$-type model shown in
Figure~\ref{seven}(7).

This proves the lemma.
\end{proof}

\subsection{The seven model types}

\begin{lemma}[Classification assuming no proper non-planar subgraph]\label{lem:seven-reduction-types}
Let
$X=(G\times S^1)\cup H$
be a critical complex for $S^3$, and let
$\widehat X=G\cup H$
be its reduction.
Assume that $\widehat X$ has no proper non-planar subgraph.
Then $\widehat X$ is homeomorphic to one of the seven graphs shown in Figure~1.

More precisely, exactly two of them belong to the contraction class represented by
Figure~3(i), exactly four belong to the contraction class represented by
Figure~3(ii), and exactly one belongs to the contraction class represented by
Figure~3(iii).
\end{lemma}

\begin{proof}
By Lemma~\ref{forest}, $H$ is a forest.
By Lemma~\ref{lem:reduction-nonplanar}, $\widehat X$ is non-planar.
By Lemma~\ref{lem:G-planar}, $G$ is planar.
Since $\widehat X$ has no proper non-planar subgraph, Lemma~\ref{kuratowski-core}
implies that $\widehat X$ is homeomorphic to a subdivision of either $K_5$ or
$K_{3,3}$.

We treat the two cases separately.

\smallskip
\noindent
\textit{Case 1. $\widehat X$ is homeomorphic to a subdivision of $K_5$.}

Suppress all degree-two vertices of $\widehat X$, and let $\overline H$ be the image
of $H$ in the resulting copy of $K_5$.
Put
$\overline G:=K_5-E(\overline H)$.
Since $H$ is a forest, $\overline H$ is a forest as well.
Since $G$ is planar, we must have $\overline H\neq\emptyset$.

Up to automorphism of $K_5$, there are exactly nine possibilities for the nonempty
forest $\overline H$:
\[
K_2,\ L_3,\ 2K_2,\ K_{1,3},\ L_4,\ L_3\sqcup K_2,\ K_{1,4},\ T_5,\ L_5.
\]

By Lemma~\ref{lem:K5-case-check}, the four types
\[
2K_2,\qquad L_3\sqcup K_2,\qquad T_5,\qquad L_5
\]
give embeddable complexes and hence cannot occur for a critical complex, while the
three types
\[
K_2,\qquad L_3,\qquad K_{1,3}
\]
give non-embeddable but non-critical complexes and hence cannot occur either.
Therefore only
$\overline H\cong K_{1,4}$
or
$\overline H\cong L_4$
remain.
Again by Lemma~\ref{lem:K5-case-check}, undoing the suppressed degree-two vertices
yields exactly the six graphs Figure~\ref{seven}(1)--Figure~\ref{seven}(6).

\smallskip
\noindent
\textit{Case 2. $\widehat X$ is homeomorphic to a subdivision of $K_{3,3}$.}

Suppress all degree-two vertices of $\widehat X$, and let $\overline H$ be the image
of $H$ in the resulting copy of $K_{3,3}$.
Put
$\overline G:=K_{3,3}-E(\overline H)$.
As above, $\overline H$ is a nonempty forest.

By Lemma~\ref{lem:K33-case-check}, the possible forest patterns
$\overline H\subset K_{3,3}$ have been exhausted up to automorphism.
All cases except the $K_{1,3}$ pattern are either embeddable, or contain
a proper non-embeddable subcomplex, or remain non-embeddable after an
elementary puncturing. Hence none of them can occur for a critical complex.
Thus the only remaining possibility is
\[
  \overline H\cong K_{1,3}.
\]
Undoing the suppressed degree-two vertices yields exactly one graph, namely
Figure~\ref{seven}(7).

Combining the two cases, we conclude that $\widehat X$ is homeomorphic to one of
the seven graphs shown in Figure~\ref{seven}. Moreover, two of them belong to the
contraction class represented by Figure~\ref{3types}(i), four belong to the
contraction class represented by Figure~\ref{3types}(ii), and one belongs to the
contraction class represented by Figure~\ref{3types}(iii).
\end{proof}

\begin{lemma}[Finite obstruction lemma]\label{lem:finite-obstruction-new}
Let $X=(G\times S^1)\cup H$ be critical for $S^3$, and let
$\widehat X=G\cup H$ be its reduction. Let
$\Gamma\subsetneq\widehat X$ be an inclusion-minimal non-planar subgraph, and write
$\Gamma=G_\Gamma\cup H_\Gamma$ and
$X_\Gamma=(G_\Gamma\times S^1)\cup H_\Gamma$.
Assume that $X_\Gamma$ embeds in $S^3$, and that there is no proper intermediate
non-embeddable subcomplex
$X_\Gamma\subset Y\subsetneq X$ of the form
$Y=(G_Y\times S^1)\cup H_Y$.
Then $\widehat X$ is one of the seven reduction graphs shown in
Figure~\ref{seven}.
\end{lemma}

\begin{proof}
Since $\Gamma$ is inclusion-minimal non-planar, Lemma~\ref{kuratowski-core} implies
that $\Gamma$ is a subdivision of either $K_5$ or $K_{3,3}$. Suppress the degree-two
vertices in this Kuratowski core. The image of the $H$-part is a forest by
Lemma~\ref{forest}, and the image of the $G$-part is planar by
Lemma~\ref{lem:G-planar}.

The hypothesis that no proper intermediate complex is non-embeddable means that
the first obstruction which appears when the remaining sectors and graph edges are
added to $X_\Gamma$ is the whole complex $X$ itself. By
Lemma~\ref{lem:region-incidence-puncture}, such an obstruction is purely a
face-incidence obstruction for a forest attached to a planar product part.

If the suppressed core is $K_5$, Lemma~\ref{lem:K5-case-check} gives the complete
finite list. The embeddable cases cannot occur; the non-embeddable but
non-critical cases cannot occur because an elementary puncturing leaves a
non-embeddable complex; hence only $\overline H\cong K_{1,4}$ or
$\overline H\cong L_4$ remains. These give exactly the six reductions
Figure~\ref{seven}(1)--(6).

If the suppressed core is $K_{3,3}$, Lemma~\ref{lem:K33-case-check} gives the
corresponding finite list. The only remaining critical candidate is
$\overline H\cong K_{1,3}$, which gives exactly Figure~\ref{seven}(7).
Therefore $\widehat X$ is one of the seven reductions.
\end{proof}

\begin{proposition}[$\widehat X$ has no proper non-planar subgraph]\label{prop:no-proper-nonplanar-new}
If $X=(G\times S^1)\cup H$ is critical for $S^3$, then the reduction graph
$\widehat X=G\cup H$ has no proper non-planar subgraph.
\end{proposition}

\begin{proof}
Assume that $\widehat X$ has a proper non-planar subgraph. Choose an
inclusion-minimal non-planar subgraph $\Gamma\subsetneq\widehat X$ and set
$X_\Gamma=(G_\Gamma\times S^1)\cup H_\Gamma$.
Since isolated vertices do not affect planarity, we may choose $\Gamma$ so that
at least one edge of $\widehat X$ is not contained in $\Gamma$.

If $X_\Gamma$ is non-embeddable in $S^3$, then Lemma~\ref{lem:choose-simplex-away}
provides a simplex $\sigma$ of $X''$ such that
$|X_\Gamma|\subset |X''|\setminus st(\sigma)$. This contradicts the criticality of
$X$. Hence $X_\Gamma$ embeds in $S^3$.

If there were a proper intermediate non-embeddable subcomplex
$X_\Gamma\subset Y\subsetneq X$ of the form $(G_Y\times S^1)\cup H_Y$, then the
same lemma would place $|Y|$ inside $|X''|\setminus st(\sigma)$ for a suitable
simplex $\sigma$, again contradicting criticality. Hence no such intermediate
subcomplex exists.

By Lemma~\ref{lem:finite-obstruction-new}, the reduction graph $\widehat X$ is one
of the seven graphs in Figure~\ref{seven}. Each of these seven graphs is an
inclusion-minimal non-planar graph, being a subdivision of $K_5$ or $K_{3,3}$.
This contradicts the assumption that $\Gamma$ was a proper non-planar subgraph of
$\widehat X$. Therefore $\widehat X$ has no proper non-planar subgraph.
\end{proof}

In the following, we record a convenient reformulation of the seven explicit
critical complexes of Figure~\ref{seven}.
Motivated by the families constructed in Section~3 and by Theorem~\ref{product},
we group these seven complexes into three combinatorial types.

Throughout this subsection, we assume that every sector of $M$ is orientable, that
every attaching map has degree one, and that $M$ contains no non-orientable
sector.

\begin{definition}
A complex $X=M\cup G$ is said to be of \emph{$K_4$-type}, \emph{$\Theta_4$-type},
or \emph{$K_{2,3}$-type} if it satisfies one of the following conditions.
\end{definition}

\begin{enumerate}[(I)]
\item \textbf{$K_4$-type.}
The branch set $B(M)$ consists of four components
$B_1,B_2,B_3,B_4,$
and the sector set consists of six annuli
$S_{ij}\qquad (1\le i<j\le 4),$
with
$\partial S_{ij}=B_i\cup B_j .$
Thus the incidence pattern of branches and sectors is that of the edges and
vertices of $K_4$.
Moreover, $M$ embeds in $S^3$ so that its complement consists of four regions
$R_1,R_2,R_3,R_4,$
where
\[
\partial R_k=\bigcup_{\substack{1\le i<j\le 4\\ i,j\ne k}} S_{ij}.
\]

The graph $G$ is either a copy of $K_{1,4}$ or an H-shaped tree, and its four
degree-one vertices are attached to points of $B_1,B_2,B_3,B_4$, respectively.

Items (1) and (2) of Figure~\ref{seven} are of this type.

\begin{figure}[htbp]
	\begin{center}
	\includegraphics[trim=0mm 0mm 0mm 0mm, width=.4\linewidth]{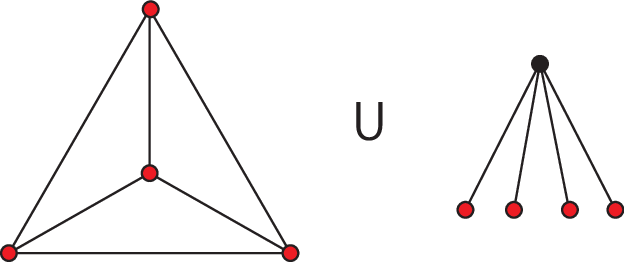}
	\end{center}
	\caption{$K_4$-type.}
	\label{K4-type}
\end{figure}

\item \textbf{$\Theta_4$-type.}
The branch set $B$ consists of two components
and the sector set consists of four sectors
$S_0,S_1,S_2,S_3,$
each satisfying
$\partial S_i=B.$
The multibranched surface $M$ embeds in $S^3$ so that its complement consists of
four regions
$R_1,R_2,R_3,R_4,$
with
\[
\partial R_1=S_0\cup S_1,\quad
\partial R_2=S_1\cup S_2,\quad
\partial R_3=S_2\cup S_3,\quad
\partial R_4=S_3\cup S_0.
\]
Equivalently, the rotation system of $M$ agrees with that of $\Theta_4$.

The graph $G$ consists of three additional edges attached to the sectors of $M$
in one of the four attachment patterns represented by
Figure~\ref{seven}(3), Figure~\ref{seven}(4), Figure~\ref{seven}(5), and
Figure~\ref{seven}(6).

Items (3), (4), (5), and (6) of Figure~\ref{seven} are of this type.

\begin{figure}[htbp]
	\begin{center}
	\includegraphics[trim=0mm 0mm 0mm 0mm, width=.4\linewidth]{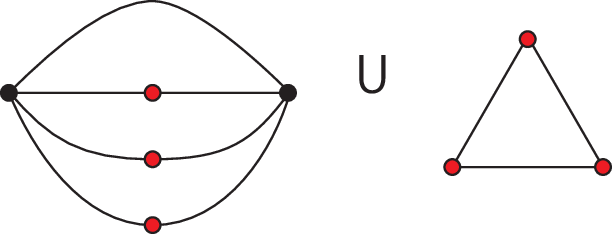}
	\end{center}
	\caption{$\Theta_4$-type.}
	\label{Theta4-type}
\end{figure}

\item \textbf{$K_{2,3}$-type.}
The branch set $B$ consists of two components
and the sector set consists of three sectors
$S_1,S_2,S_3,$
each satisfying
$\partial S_i=B.$
The multibranched surface $M$ embeds in $S^3$ so that its complement consists of
three regions
$R_1,R_2,R_3,$
with
\[
\partial R_1=S_1\cup S_2,\quad
\partial R_2=S_2\cup S_3,\quad
\partial R_3=S_3\cup S_1.
\]

The graph $G$ is isomorphic to $K_{1,3}$, and its three leaves are attached to
points in $\operatorname{int}S_1$, $\operatorname{int}S_2$, and
$\operatorname{int}S_3$, respectively.

Item (7) of Figure~\ref{seven} is of this type.

\begin{figure}[htbp]
	\begin{center}
	\includegraphics[trim=0mm 0mm 0mm 0mm, width=.4\linewidth]{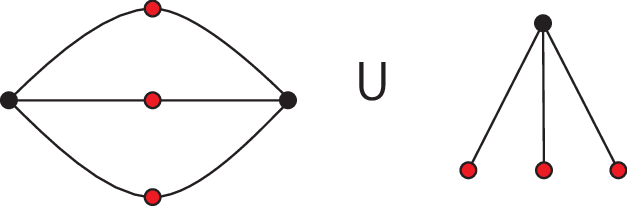}
	\end{center}
	\caption{$K_{2,3}$-type.}
	\label{K23-type}
\end{figure}
\end{enumerate}

The next proposition shows that each of these model types is indeed critical for
$S^3$.

\begin{proposition}\label{three-types-critical}
Every complex of $K_4$-type, $\Theta_4$-type, or $K_{2,3}$-type is critical for
$S^3$.
\end{proposition}

\begin{proof}
We first prove non-embeddability.

\medskip
\noindent
\textit{$K_4$-type.}
Let $X=M\cup G$ be of $K_4$-type.
Remove from $G$ the four attaching vertices on the branch components
$B_1,B_2,B_3,B_4$.
The remaining graph, denoted by $G^\circ$, is connected; indeed, $G$ is either a
copy of $K_{1,4}$ or an H-shaped tree.
Suppose that $X$ embeds in $S^3$.

Since $G^\circ$ is connected and disjoint from $M$, its image must lie in the
closure of a single complementary region of $M$.
By definition, the complement of $M$ consists of four regions
$R_1,R_2,R_3,R_4,$
where $R_k$ is incident exactly to the three branch components
$B_i$ with $i\ne k$.
In particular, no complementary region is incident to all four branch components.

On the other hand, the closure of $G^\circ$ in $X$ has four ends attached
respectively to $B_1,B_2,B_3,B_4$.
Hence the complementary region containing $G^\circ$ would have to be incident to
all four branch components, which is impossible.
Therefore $X$ is not embeddable in $S^3$.

\medskip
\noindent
\textit{$\Theta_4$-type.}
Let $X=M\cup G$ be of $\Theta_4$-type.
By definition, $M$ has four sectors
$S_0,S_1,S_2,S_3$
and four complementary regions
$R_1,R_2,R_3,R_4$
with
\[
\partial R_1=S_0\cup S_1,\quad
\partial R_2=S_1\cup S_2,\quad
\partial R_3=S_2\cup S_3,\quad
\partial R_4=S_3\cup S_0.
\]
Thus the closure of each complementary region is incident to exactly one adjacent
pair of sectors.

Suppose that $X$ embeds in $S^3$.
Then each connected component of $G\setminus (G\cap M)$ is contained in the closure
of a single complementary region of $M$.
Hence, if $C\subset G$ is any connected subgraph, then the set
$A(C):=\{\,S_i \mid C\cap S_i\neq \emptyset\,\}$
must be contained in one of the four adjacent pairs
$\{S_0,S_1\},\ \{S_1,S_2\},\ \{S_2,S_3\},\ \{S_3,S_0\}.$
We now check that this fails in each of the four models of
Figure~\ref{seven}(3)--(6).

\smallskip
\noindent
For Figure~\ref{seven}(3), let $C_3$ be the connected subgraph consisting of the
three red edges.  Then $C_3$ meets the three sectors $S_0,S_1,S_2$, so
$A(C_3)=\{S_0,S_1,S_2\}.$
This set is not contained in any adjacent pair, contradiction.

\smallskip
\noindent
For Figure~\ref{seven}(4), let $C_4$ be the connected red path with four attaching
vertices.  It meets all four sectors, and hence
$A(C_4)=\{S_0,S_1,S_2,S_3\}.$
Again this is not contained in any adjacent pair, contradiction.

\smallskip
\noindent
For Figure~\ref{seven}(5), let $C_5$ be the middle red edge.  Its two endpoints lie
on opposite sectors, say $S_0$ and $S_2$.  Therefore
$A(C_5)=\{S_0,S_2\},$
which is not contained in any adjacent pair, contradiction.

\smallskip
\noindent
For Figure~\ref{seven}(6), let $C_6$ be the horizontal red edge.  Its two endpoints
also lie on opposite sectors, say $S_0$ and $S_2$.  Thus
$A(C_6)=\{S_0,S_2\},$
which is not contained in any adjacent pair, contradiction.

Therefore none of the four $\Theta_4$-type complexes can be embedded in $S^3$.

\medskip
\noindent
\textit{$K_{2,3}$-type.}
Let $X=M\cup G$ be of $K_{2,3}$-type.
Remove from $G$ the three attaching vertices on
$\operatorname{int}S_1$, $\operatorname{int}S_2$, and $\operatorname{int}S_3$.
The remaining graph $G^\circ$ is connected, since $G\cong K_{1,3}$.

Suppose that $X$ embeds in $S^3$.
Since $G^\circ$ is connected and disjoint from $M$, its image must lie in the
closure of a single complementary region of $M$.
By definition, the complement of $M$ consists of three regions
$R_1,R_2,R_3$
with
\[
\partial R_1=S_1\cup S_2,\quad
\partial R_2=S_2\cup S_3,\quad
\partial R_3=S_3\cup S_1.
\]
Thus each complementary region is incident to exactly two of the three sectors.

However, the closure of $G^\circ$ in $X$ has three ends attached respectively to
$\operatorname{int}S_1$, $\operatorname{int}S_2$, and $\operatorname{int}S_3$.
Hence the complementary region containing $G^\circ$ would have to be incident to
all three sectors, which is impossible.
Therefore $X$ is not embeddable in $S^3$.

\medskip
We next verify minimality.
Let $\sigma$ be a simplex of the second barycentric subdivision $X''$.
By Remark~\ref{rem:criticality}, up to PL homeomorphism, deleting $st(\sigma)$ is a
local elementary deletion.
Accordingly, we distinguish four cases.

\medskip
\noindent
\textit{Case 1. $st(\sigma)$ is contained in the interior of a sector of $M$.}

For a $K_4$-type complex, deleting a small open disk from a sector
$S_{ij}$ merges the two complementary regions on opposite sides of $S_{ij}$.
If $\{i,j,k,\ell\}=\{1,2,3,4\}$, these are precisely the regions $R_k$ and
$R_\ell$.
Their union is incident to all four branch components
$B_1,B_2,B_3,B_4$.
Hence the whole graph part can be embedded in this merged region, so the resulting
space embeds in $S^3$.

For a $\Theta_4$-type complex, suppose that the deleted disk lies in a sector
$S_i$ (indices taken modulo $4$). Then the two complementary regions adjacent to
$S_i$ merge into a single region $Q$, while the other two complementary regions
remain unchanged. The region $Q$ is incident to the three consecutive sectors
$S_{i-1}, S_i, S_{i+1}$. Since the graph part consists of three edges attached in one
of the four patterns of Figure~\ref{seven}(3)--(6), the edge or edges incident to $S_i$ can be
placed in $Q$, and every remaining edge can be placed in the unchanged
complementary region incident to its two attaching sectors. Hence the graph part
can be embedded in the complement of the modified surface, so the resulting space
embeds in $S^3$.

For a $K_{2,3}$-type complex, deleting a small open disk from one of the sectors
merges the two complementary regions adjacent to that sector.
The merged region is incident to all three sectors
$S_1,S_2,S_3$, and therefore the whole graph $K_{1,3}$ can be embedded there.
Thus the resulting space embeds in $S^3$.

\medskip
\noindent
\textit{Case 2. $st(\sigma)$ meets a branch component of $M$.}

For a $K_4$-type complex, deleting a small neighborhood on a branch component
$B_i$ opens that branch and makes the complementary regions adjacent to $B_i$
communicate through the deleted neighborhood.
The resulting complement contains a region incident to all four branch components,
so the graph part can be embedded there.
Hence the resulting space embeds in $S^3$.

For a $\Theta_4$-type complex, the branch set has two components.
Deleting a small neighborhood on either branch component opens the cyclic arrangement of the four
regions around that branch component.
In each of the four explicit models, the graph part can then be embedded in the
resulting connected region.
Thus the resulting space embeds in $S^3$.

For a $K_{2,3}$-type complex, the branch set has two components.
Deleting a small neighborhood on either branch component makes the three complementary regions communicate through the
deleted neighborhood.
The resulting complement contains a region incident to all three sectors, so the
graph part can be embedded there.
Hence the resulting space embeds in $S^3$.

\medskip
\noindent
\textit{Case 3. $st(\sigma)$ is contained in the graph part $G$ away from the
attaching points.}

For a $K_4$-type complex, such a deletion removes a nonempty open subarc from
the graph part. Each connected component of the remaining graph is attached to a
proper subset of the four branch components $B_1,B_2,B_3,B_4$, and hence to at
most three of them. Since, for every three-element subset of
$\{B_1,B_2,B_3,B_4\}$, there is a complementary region of $M$ incident exactly to
those three branch components, each component of the remaining graph can be
embedded in a suitable complementary region. Thus the resulting space embeds in
$S^3$.

For a $\Theta_4$-type complex, such a deletion removes a nonempty open subarc
from one of the three graph edges. Hence the remaining graph part is a forest whose
connected components are attached to at most two consecutive sectors. Since each
complementary region of $M$ is incident to exactly two consecutive sectors, each
component can be embedded in a suitable complementary region. Thus the resulting
space embeds in $S^3$.

For a $K_{2,3}$-type complex, such a deletion removes a nonempty open subarc from
$K_{1,3}$.
Every connected component of the remaining graph is attached to at most two of the
three sectors.
Since each complementary region of $M$ is incident to exactly two sectors, each
component can be embedded in a suitable complementary region.
Thus the resulting space embeds in $S^3$.

\medskip
\noindent
\textit{Case 4. $st(\sigma)$ contains an attaching point of $G$ on $M$.}

For a $K_4$-type complex, deleting such a neighborhood removes one of the four
incidences of the graph part with the branch components. Consequently each
connected component of the remaining graph is attached to at most three of
$B_1,B_2,B_3,B_4$. As above, any such subset is incident to some complementary
region of $M$, and so each component can be embedded in a suitable complementary
region. Hence the resulting space embeds in $S^3$.

For a $\Theta_4$-type complex, deleting such a neighborhood removes one
incidence between the graph part and the sector set. Consequently each connected
component of the remaining graph part is attached to at most two consecutive
sectors. Since each complementary region of $M$ is incident to exactly two
consecutive sectors, each component can be embedded in a suitable complementary
region. Hence the resulting space embeds in $S^3$.

For a $K_{2,3}$-type complex, deleting such a neighborhood removes one of the
three incidences with the sectors.
The remaining graph is attached to at most two sectors, hence it embeds in a
suitable complementary region of $M$.

Therefore in every case the space $|X''|\setminus st(\sigma)$ embeds in $S^3$.
Since $X$ itself is not embeddable in $S^3$, every complex of $K_4$-type,
$\Theta_4$-type, or $K_{2,3}$-type is critical for $S^3$.
\end{proof}

\subsection{Proof of the classification theorem}

\begin{theorem}[Classification theorem]\label{product}
Let \(X\) be a critical complex for \(S^3\) of the form
\(X=(G\times S^1)\cup H\), where \(G\) and \(H\) are graphs.
Then \(|X|\) is homeomorphic to one of the seven complexes shown in
Figure~\ref{seven}. Conversely, each of the seven complexes in
Figure~\ref{seven} is critical for \(S^3\).
\end{theorem}

Equivalently, $X$ belongs to exactly one of the following three families:
\begin{enumerate}
\item[(i)] \textbf{$K_4$-type:}
      $G\cong K_4$, and $H$ is either $K_{1,4}$ or an H-shaped tree (that is, the tree obtained by joining the centers of two copies of $K_{1,3}$ by a single edge),
with the degree-one vertices of $H$ attached to branch circles of $G\times S^1$;
\item[(ii)] \textbf{$\Theta_4$-type:}
the multibranched surface part is of $\Theta_4$-type, and $H$ consists of three
additional edges such that the reduction graph belongs to the contraction class
represented by Figure~\ref{3types}(ii);
\item[(iii)] \textbf{$K_{2,3}$-type:}
      $G\cong K_{2,3}$, and $H\cong K_{1,3}$, with the three leaves attached to
      three distinct sectors of $G\times S^1$.
\end{enumerate}
Conversely, each of the seven complexes in Figure~\ref{seven} belongs to one of
the three model types formalized above, and
Proposition~\ref{three-types-critical} shows that each of them is critical for
$S^3$.
Hence the list is complete.

\begin{proof}
Let $\widehat X=G\cup H$ be the reduction of $X$.
By Proposition~\ref{prop:no-proper-nonplanar-new}, $\widehat X$ has no proper non-planar subgraph.
Hence, by Lemma~\ref{lem:seven-reduction-types}, $\widehat X$ is homeomorphic to one of the
seven graphs shown in Figure~\ref{seven}.

We now recover $X$ from $\widehat X$ by thickening each edge of the $G$-part to an
annulus, while keeping the $H$-part $1$-dimensional.

\medskip
\noindent
\textit{Case 1.}
Suppose that $\widehat X$ belongs to the contraction class represented by
Figure~\ref{3types}(i).
Then the $G$-part is homeomorphic to $K_4$, and the $H$-part is a tree with four
terminal vertices attached to the four vertices of the $K_4$-part.
After thickening the $G$-part, we obtain $K_4\times S^1$, and the four terminal
vertices of $H$ attach to the four branch circles.
Up to homeomorphism, there are exactly two possibilities for the forest part:
either $H\cong K_{1,4}$ or $H$ is an H-shaped tree.
These are precisely the two complexes shown in Figure~\ref{seven}(1) and
Figure~\ref{seven}(2).
Hence $X$ is of $K_4$-type.

\medskip
\noindent
\textit{Case 2.}
Suppose that \(\widehat X\) belongs to the contraction class
represented by Figure~\ref{3types}(ii).
Then the \(G\)-part is homeomorphic to the \(\Theta_4\)-graph,
that is, the graph consisting of two vertices joined by four edges,
and the \(H\)-part consists of three additional edges attached as in
Figure~\ref{3types}(ii).
After thickening the \(G\)-part, we obtain the multibranched surface
of \(\Theta_4\)-type.
Up to homeomorphism, there are exactly four possible attachment
patterns, namely the four complexes shown in
Figure~\ref{seven}(3)--(6).
Hence \(X\) is of \(\Theta_4\)-type.

\medskip
\noindent
\textit{Case 3.}
Suppose that $\widehat X$ belongs to the contraction class represented by
Figure~\ref{3types}(iii).
Then the $G$-part is homeomorphic to $K_{2,3}$ and the $H$-part is homeomorphic
to $K_{1,3}$.
After thickening the $G$-part, we obtain $K_{2,3}\times S^1$, and the three
leaves of the $K_{1,3}$-part attach to three distinct sectors.
Up to homeomorphism, this yields a unique complex, namely the one shown in
Figure~\ref{seven}(7).
Hence $X$ is of $K_{2,3}$-type.

\medskip
Therefore $|X|$ is homeomorphic to one of the seven complexes shown in
Figure~\ref{seven}.
Equivalently, $X$ belongs to exactly one of the three families listed in the
statement.

Conversely, each of the seven complexes in Figure~\ref{seven} belongs to one of
the three families above, and Proposition~\ref{three-types-critical} shows that
each of them is critical for $S^3$.
Hence the list is complete.
\end{proof}

\bigskip

\noindent{\bf Declaration of competing interest.}

The authors declare that they have no competing interests.

\bigskip

\noindent{\bf Declaration of generative AI and AI-assisted technologies in the manuscript preparation process.}

During the preparation of this work the second author used ChatGPT (OpenAI) for
assistance with language editing, formatting, and consistency checking. After
using this tool, the authors reviewed and edited the content as needed and take
full responsibility for the content of the article.

\bigskip

\noindent{\bf Acknowledgement.}

This work was carried out while the second author was visiting the Instituto de Matem\'{a}ticas, Universidad Nacional Aut\'onoma de M\'exico, from April 2022 to March 2023.
The second author would like to express his sincere gratitude to the first author for his hospitality.
The authors would also like to thank Tatsuya Tsukamoto for helpful comments.

\bibliographystyle{amsplain}

\end{document}